\documentclass[3p,preprint]{elsarticle}

\usepackage[utf8]{inputenc}
\usepackage[T1]{fontenc}

\usepackage{amsmath,amssymb,amsthm,mathtools}
\usepackage{bm}

\usepackage{graphicx}
\usepackage{booktabs}
\usepackage{placeins}
\usepackage{xcolor}

\graphicspath{{figures/}}

\biboptions{sort&compress}

\newtheorem{theorem}{Theorem}[section]
\newtheorem{proposition}[theorem]{Proposition}
\newtheorem{lemma}[theorem]{Lemma}

\newtheorem{assumption}[theorem]{Assumption}

\theoremstyle{definition}
\newtheorem{definition}[theorem]{Definition}

\theoremstyle{remark}
\newtheorem{remark}[theorem]{Remark}

\newenvironment{acknowledgements}
  {\section*{Acknowledgements}}
  {}

\newcommand{\R}{\mathbb{R}}
\newcommand{\N}{\mathbb{N}}
\newcommand{\wt}{\widetilde}

\newcommand{\norm}[1]{\left\lVert #1\right\rVert}
\newcommand{\abs}[1]{\left\lvert #1\right\rvert}

\newcommand{\Pd}{\Pi_d(\R^s)}
\newcommand{\Ind}[1]{I_{#1}}
\newcommand{\Kop}{\wt K}

\newcommand{\dist}{\operatorname{dist}}
\newcommand{\diam}{\operatorname{diam}}
\newcommand{\intr}{\operatorname{int}}

\newcommand{\Einf}{E_{\infty}}

\begin{document}

\begin{frontmatter}

\title{Data-dependent Compact Multinode Shepard Method}

\author[address-CS]{Francesco Dell'Accio}
\ead{francesco.dellaccio@unical.it}

\author[address-CS]{Filomena Di Tommaso}
\ead{filomena.ditommaso@unical.it}

\author[address-CS]{Francesco Larosa\corref{corrauthor}}
\ead{francesco.larosa@unical.it}
\cortext[corrauthor]{Corresponding author}

\author[address-cartagena]{Juan Ruiz-\'Alvarez}
\ead{juan.ruiz@upct.es}

\author[address-valencia]{Dionisio F. Y\'a\~nez}
\ead{dionisio.yanez@uv.es}

\address[address-CS]{
Department of Mathematics and Computer Science,
University of Calabria,
via P. Bucci cubo 30 A,
87036 Rende (CS), Italy
}

\address[address-cartagena]{
Departamento de Matem\'atica Aplicada y Estad\'istica,
Universidad Polit\'ecnica de Cartagena,
Spain
}

\address[address-valencia]{
Departamento de Matem\'aticas,
Universidad de Valencia,
Spain
}

\begin{abstract}
We introduce a data-dependent compact multinode Shepard method for the approximation of scattered multivariate data. The proposed approach extends the compact multinode Shepard operator by incorporating nonlinear, data-dependent weights inspired by the weighted essentially non-oscillatory (WENO) methodology while preserving the partition of unity structure and the locality of the approximation. The construction is based on compactly supported multinode weight functions associated with local polynomial interpolants defined on unisolvent stencils. A smoothness indicator is employed to adapt the contribution of each stencil according to the local regularity of the underlying function, thereby reducing spurious oscillations near discontinuities while retaining the optimal approximation order in smooth regions. We establish the theoretical approximation properties of the proposed operator: it achieves the order of accuracy of its linear counterpart for sufficiently smooth functions and, under an explicit indicator-separation hypothesis, the same order is retained at every covered evaluation point near a jump discontinuity. Numerical experiments for local degrees up to four illustrate the performance of the method on both smooth and non-smooth data, on scattered as well as on gridded nodes.
\end{abstract}


\end{frontmatter}

\section{Introduction}\label{sec:introduction}
Scattered data approximation deals with the problem of reconstructing
an unknown function from values prescribed at irregularly distributed
data sites. The absence of an underlying grid or mesh makes this
setting particularly attractive in applications where data are
collected at irregularly distributed locations and no natural mesh is
available. Among the several techniques developed for this purpose,
Shepard-type methods have retained considerable interest because of
their simple structure, their meshfree character, and the ease with
which they can be applied to arbitrarily scattered data. The classical
Shepard method \cite{Shepard} constructs the approximant as a
normalized weighted average of the data values, using inverse powers
of the distances from the evaluation point. Its partition of unity
structure immediately provides  stability
properties,  and the
singularity of the weights at the nodes makes it interpolatory.
However, the fact that only constants are reproduced limits the
approximation order that can be achieved.

A considerable part of the subsequent development of Shepard-type methods has therefore been devoted to improving the accuracy and the locality of the original construction while preserving its basic simplicity. Compactly supported weights were introduced to restrict the influence of distant data sites and to obtain local
approximations \cite{FrankeNielson}. A different direction, originating
from the triangular Shepard construction, consists in replacing the
individual data values by polynomial interpolants based on small
subsets of the nodes. This idea leads to the multinode Shepard method
\cite{DRNArates,HexagonalShepard,TriangularShepard}, in which local
polynomial interpolants are blended through a partition of unity
generated by products of inverse distances. By increasing the degree
of the local polynomials, the method retains the interpolatory
character of the classical Shepard construction while achieving higher
approximation orders on sufficiently smooth functions. More recently,
the same multinode principle has been combined with compactly supported
weights in the interpolation of scattered data on the torus
\cite{TorusCMS}, providing a local version of the multinode
construction.

The use of higher-degree local interpolants, however, also makes the
quality of the approximation increasingly dependent on the regularity
of the data contained in each stencil. When the underlying function is
only piecewise smooth, a stencil may contain data belonging to
different smooth branches. The corresponding polynomial interpolant
can then be a poor local representation of either branch and may
produce pronounced oscillatory effects in a neighborhood of the
discontinuity. This difficulty is not specific to Shepard
approximation and is closely related to the general problem of
approximating piecewise smooth functions from scattered data; see, for
example, \cite{ArandigaCohen,BozziniRossini,LenarduzziSchaback}. In the
linear multinode Shepard construction all the stencils whose supports
contain the evaluation point (the \emph{active} stencils,
Section~\ref{sec:multinode}) contribute
according to geometrically determined weights, independently of the
regularity of the data from which their local interpolants are built.
It is therefore natural to ask whether the partition of unity
mechanism can be made sensitive to the local behavior of the data, so
that contributions associated with nonsmooth configurations are
reduced without affecting the approximation in smooth regions.

This is the point of view adopted in the present paper. Our
construction is inspired by the weighted essentially non-oscillatory
(WENO) methodology \cite{Shu}, whose basic principle is to assign
nonlinear weights according to suitable measures of local smoothness.
WENO-type modifications of partition of unity approximation have already been considered for B-spline bases \cite{AmatWENO} (see also \cite{AmatBspline}), while a data-dependent version of the classical Shepard method has been developed in \cite{LRRY}. Related nonlinear approaches have also been proposed for radial basis function (RBF) partition of unity approximation \cite{NonlinearPUM} and moving least squares approximation \cite{DataDependentMLS}. Here we combine this nonlinear weighting principle with the compact multinode Shepard construction. Each local polynomial interpolant is associated with a smoothness indicator, and the corresponding compact multinode weight is modified according to
the value of this indicator before the partition of unity is
renormalized. In the construction considered in this paper, the
indicator is obtained from the residuals of the local polynomial at
data sites lying outside the interpolation stencil. Since these values
are already part of the available scattered data, the indicator does
not require additional function evaluations.

The resulting operator remains local and interpolatory and preserves
the partition of unity structure of the underlying compact multinode
method. Its nonlinear character is relevant only in the relative
weight assigned to the active local interpolants. This observation is
particularly useful in smooth regions: once the local interpolants
have the expected approximation accuracy, any convex combination of
them inherits the same order. We show, in fact, that under uniform
locality and stability assumptions the data-dependent operator retains
the order $O(h^{d+1})$ of the multinode construction for sufficiently
smooth functions, where $h$ is the fill distance of the node set and $d$
the degree of the local polynomial interpolants, independently of the
nonlinear exponent. For
piecewise smooth functions separated by a jump discontinuity, the
analysis is necessarily more delicate. We identify explicit
conditions under which the smooth and contaminated contributions are
separated by the indicator and prove that the same approximation order
is recovered at evaluation points satisfying these conditions. At the same time, we show by an explicit configuration that the residual indicator does not detect every stencil crossing a discontinuity. The convergence analysis requires that at least one suitable active stencil lying on one side of the discontinuity has a small smoothness indicator, while all active stencils crossing the discontinuity have indicators uniformly bounded away from zero. We refer to this requirement as indicator separation and impose it explicitly, since it is not guaranteed by the definition of the residual indicator.

The theoretical analysis is complemented by numerical experiments for
local polynomial degrees up to four. The experiments compare the
linear and data-dependent compact multinode Shepard operators on both
smooth and piecewise smooth data and investigate the influence of the
nonlinear exponent. 

The paper is organized as follows. Section~\ref{sec:shepard} recalls
the classical Shepard method and its compactly supported version.
Section~\ref{sec:multinode} reviews the multinode Shepard construction
and introduces the compact multinode operator used throughout the
paper. In Section~\ref{sec:construction} we define the data-dependent
compact multinode Shepard method and the residual-based smoothness
indicator. Section~\ref{sec:properties} is devoted to the theoretical
properties of the new operator, including the behavior of the
indicator, the structural properties of the nonlinear construction,
and the convergence analysis in smooth regions and near a
discontinuity. The numerical experiments are presented in
Section~\ref{sec:numerics}, and concluding remarks are collected in
Section~\ref{sec:conclusions}.

\section{Classical and compact Shepard methods}\label{sec:shepard}

Following the notation of \cite{DellAccioSoftware}, let $\Omega$ be an open, connected and bounded subset of $\R^s$, $s\in\N$, and let $X=\{x_1,\dots,x_n\}\subset\Omega$ be a set of nodes. Given a function $f:\Omega\to\R$ known only through the values $f_i=f(x_i)$ at the nodes, the classical Shepard method \cite{Shepard} is defined by
\begin{equation}\label{eq:shepard}
S_\mu[f](x)=\sum_{i=1}^{n}A_{\mu,i}(x)\,f_i,\qquad x\in\Omega\setminus X,
\end{equation}
where, for a fixed parameter $\mu>0$, the Shepard weights $A_{\mu,i}$, $i=1,\dots,n$, are given by
\begin{equation}\label{eq:Amui}
A_{\mu,i}(x)=
\frac{\norm{x-x_i}^{-\mu}}
{\displaystyle\sum_{k=1}^{n}\norm{x-x_k}^{-\mu}},
\qquad x\in\Omega\setminus X.
\end{equation}
Here $\norm{\cdot}$ denotes the Euclidean norm. At the data sites, the
operator is defined by continuous extension, $S_\mu[f](x_i)=f_i$, so that the method is interpolatory.
The operator $S_\mu[f]$ is a partition of unity method (PUM), see e.g. \cite{Fasshauer}, whose local approximants are the constant values $f_i$, that is, polynomials of degree zero. The classical Shepard weights satisfy the standard partition of unity
and interpolation properties established in
\cite{Barnhill,Shepard}, which we recall below.
\begin{proposition}\label{prop:A}
The functions $A_{\mu,i}$ defined in \eqref{eq:Amui} satisfy
\begin{equation}\label{eq:Aprops}
A_{\mu,i}(x)\ge0,\quad i=1,\dots,n,\qquad
\sum_{i=1}^{n}A_{\mu,i}(x)=1,\quad x\in\Omega\setminus X,
\end{equation}
and
\[
\lim_{x\to x_j}A_{\mu,i}(x)=\delta_{ij},\qquad
D^\alpha A_{\mu,i}(x_j)=0,
\quad 1\le |\alpha|<\mu,\quad i,j\in\{1,\dots,n\},
\]
where $\delta_{ij}$ is the Kronecker delta, $\alpha=(\alpha_1,\dots,\alpha_s)\in\N^s$
is a multi-index with $\abs{\alpha}=\alpha_1+\dots+\alpha_s$,
$D^\alpha=\partial^{\abs{\alpha}}/\partial x_1^{\alpha_1}\cdots\partial x_s^{\alpha_s}$,
and the derivatives are understood for the continuous extensions of
the weights.
\end{proposition}
The regularity of these extensions depends on the parameter $\mu$, as
recalled in the following result; see \cite{Barnhill}.
\begin{proposition}\label{prop:smoothA}
The continuous extensions of the functions $A_{\mu,i}$ are:
$C^0$ with cusps at the nodes if $\mu<1$;
$C^0$ with corners at the nodes if $\mu=1$;
$C^{\lfloor\mu\rfloor}$ if $\mu>1$ and $\mu\notin\mathbb N$;
$C^{\mu-1}$ if $\mu>1$ is an odd integer; and
$C^\infty$ rational functions with no real poles if $\mu>1$ is an even
integer.
\end{proposition}
A compactly supported version of the Shepard method was introduced by
Franke and Nielson \cite{FrankeNielson} by replacing the global weights $A_{\mu,i}$ with
\begin{equation}\label{eq:FrankeNielson}
\wt A_{\mu,i}(x)=
\frac{\left(\dfrac{1}{\norm{x-x_i}}-\dfrac{1}{R_i}\right)_{+}^{\mu}}
{\displaystyle\sum_{k=1}^{n}
\left(\dfrac{1}{\norm{x-x_k}}-\dfrac{1}{R_k}\right)_{+}^{\mu}},
\qquad i=1,\dots,n.
\end{equation}
Here $(z)_+=\max\{z,0\}$, and the weights are defined for
\[
x\in\bigcup_{k=1}^{n}B(x_k,R_k)\setminus X,
\]
where $B(x_k,R_k)$ denotes the open Euclidean ball centered at $x_k$ with radius $R_k>0$. The support radius $R_k$ is typically chosen so that $B(x_k,R_k)$ contains a prescribed number of data sites \cite{FrankeNielson}. As in the classical case, the definition is completed at the data sites by continuous extension.

\section{Multinode Shepard method and compact multinode Shepard method}\label{sec:multinode}
Dell'Accio et al.\ in \cite{TriangularShepard,HexagonalShepard,DRNArates} extend the idea of the Shepard method using
polynomials of higher degree. Let $d\in\N$, let $\Pd$ denote the space
of polynomials of total degree at most $d$ in $s$ variables, of dimension
$m=\binom{d+s}{s}$, and recall that a set of $m$ nodes is
\emph{$\Pd$-unisolvent} if every function defined on it admits a unique
interpolant in $\Pd$. Let
\[
\mathcal S_d=\{\tau_1,\dots,\tau_r\},\qquad
\tau_j=\{x_{j_1},\dots,x_{j_m}\}\subset X,
\]
be a family of $r$ \emph{stencils}, i.e.\ pairwise distinct unordered
$m$-subsets of $X$,
each consisting of pairwise distinct nodes, such that each
$\tau_j$ is $\Pd$-unisolvent and $\mathcal S_d$ is a covering of $X$,
i.e.\ $\bigcup\limits_{j=1}^r\tau_j=X$.
The multinode Shepard method is
\begin{equation}\label{eq:multinode}
M_\mu[f](x)=\sum_{j=1}^{r}W_{\mu,j}(x)\,P_j[f](x),\qquad
x\in\Omega\setminus X,
\end{equation}
with $P_j[f]\in\Pd$ the unique interpolation polynomial on $\tau_j$
and
\begin{equation}\label{eq:Wmuj}
W_{\mu,j}(x)=
\frac{\displaystyle\prod_{i=1}^{m}\norm{x-x_{j_i}}^{-\mu}}
{\displaystyle\sum_{k=1}^{r}\prod_{i=1}^{m}\norm{x-x_{k_i}}^{-\mu}},
\qquad\mu>0,
\end{equation}
completed by continuous extension at the data sites. The basic properties of the multinode weights established in
\cite{DRNArates} are recalled in the following proposition.

\begin{proposition}\label{prop:W}
The functions $W_{\mu,j}$ satisfy, for $x\in\Omega\setminus X$:
$W_{\mu,j}(x)\ge0$; $\sum_{j=1}^{r}W_{\mu,j}(x)=1$;
$\lim_{x\to x_i}W_{\mu,j}(x)=0$ for all $j$ with $x_i\notin\tau_j$;
and $\sum_{k\in J_i}\lim_{x\to x_i}W_{\mu,k}(x)=1$, where
$J_i=\{j:1\le j\le r,\ x_i\in\tau_j\}$. Moreover, the extended
$W_{\mu,j}$ satisfy the differentiability properties of
Proposition~\ref{prop:smoothA}.
\end{proposition}

Following the idea of Franke and Nielson \cite{FrankeNielson}, a
compact multinode Shepard (CMS) operator was developed in
\cite{TorusCMS} for scattered data interpolation on the torus, viewed
as a smooth embedded manifold in $\R^3$. Here, we adopt the same
localization principle and replace the global weights with

\begin{equation}\label{eq:CMSweights}
\wt W_{\mu,j}(x)=
\frac{\displaystyle\prod_{i=1}^{m}
\left(\dfrac{1}{\norm{x-x_{j_i}}}-\dfrac{1}{R_j}\right)_{+}^{\mu}}
{\displaystyle\sum_{k=1}^{r}\prod_{i=1}^{m}
\left(\dfrac{1}{\norm{x-x_{k_i}}}-\dfrac{1}{R_k}\right)_{+}^{\mu}},
\qquad j=1,\dots,r,
\end{equation}
where truncation and exponent act factorwise in both numerator and
denominator, and $R_j>0$ is a support radius associated with the stencil
$\tau_j$, specified below. For $x\in\Omega\setminus X$, the unnormalized numerator
associated with the $j$-th stencil is positive if and only if
$\norm{x-x_{j_i}}<R_j$ for every $i=1,\dots,m$. In order to
determine the radii $R_j$ we use the fill distance
\begin{equation}\label{eq:filldistance}
h_{X,\Omega}=\sup_{x\in\Omega}\min_{x_i\in X}\norm{x-x_i},
\end{equation}
or rather a fixed \emph{strict upper bound} of it: throughout the
paper, $h$ denotes a number satisfying $h>h_{X,\Omega}$. In practice,
$h$ is chosen as a certified computable upper bound; see
Section~\ref{sec:numerics}, since the exact fill distance, defined
through a supremum over $\Omega$, is in general not directly computable. We then set
\[
R_j=\diam(\tau_j)+h,
\]
where
\[
\diam(\tau_j)
=\max_{x,y\in\tau_j}\norm{x-y}
\]
denotes the diameter of the stencil $\tau_j$, so that the support of $\wt W_{\mu,j}$ is
\[
C_j=\bigcap_{i=1}^{m}
\{y\in\R^s:\norm{y-x_{j_i}}\le R_j\},
\]
which represents the intersection of the $m$ closed Euclidean balls of radius $R_j$ centered at the nodes of the stencil $\tau_j$. This set is independent of $\mu$ and of $\Omega$, and $\intr(C_j)$ denotes its interior in $\R^s$, so that points of $\partial\Omega$ are not excluded from the supports. For $x\in\Omega$, let
\[
J_x=\{j:1\le j\le r,\ x\in\intr(C_j)\}.
\]
We say that $x$ is \emph{covered} if $J_x\neq\emptyset$, and we call
the stencils $\tau_j$ with $j\in J_x$ \emph{active} at $x$. For a
covered $x\in\Omega\setminus X$, the CMS operator is defined by
\begin{equation}\label{eq:CMS}
\wt M_\mu[f](x)
=\sum_{j\in J_x}\wt W_{\mu,j}(x)\,P_j[f](x).
\end{equation}
At the data sites, the definition is completed by continuous
extension. The covering property $\bigcup\limits_{j=1}^r\tau_j=X$,
together with the strict inequality $h>h_{X,\Omega}$, ensures that
the interiors of the supports cover $\Omega$, as established in the
following lemma.

\begin{lemma}\label{lem:coverage}
Assume that $\bigcup\limits_{j=1}^r\tau_j=X$ and $h>h_{X,\Omega}$.
Then
\[
\Omega\subseteq\bigcup_{j=1}^r\intr(C_j),
\]
and hence $J_x\neq\emptyset$ for every $x\in\Omega$. More precisely,
if $x_i$ is a nearest data site to $x$, then
\[
x\in\intr(C_j)
\]
for every stencil $\tau_j$ containing $x_i$.
\end{lemma}
\begin{proof}
Fix $x\in\Omega$ and let $x_i\in X$ be a nearest data site to $x$.
Since the stencils cover $X$, let $\tau_j$ be any stencil containing
$x_i$. For every $x_{j_q}\in\tau_j$, the triangle inequality gives
\[
\norm{x-x_{j_q}}
\le \norm{x-x_i}+\norm{x_i-x_{j_q}}
\le h_{X,\Omega}+\diam(\tau_j)
< h+\diam(\tau_j)
=R_j.
\]
Thus $x\in\intr(C_j)$. Since $x\in\Omega$ was arbitrary, the result
follows.
\end{proof}
\begin{remark}
The strict inequality $h>h_{X,\Omega}$ is essential: with $h=h_{X,\Omega}$ the points attaining
the supremum in \eqref{eq:filldistance} may fall on support
boundaries (for instance, for degree zero, $R_j=h$, and a fill-distance
maximizer is at distance at least $h$ from every node, with equality
for at least one nearest node), and the
operator would be undefined there. Any computable upper bound strictly larger than $h_{X,\Omega}$
satisfies the hypothesis. If the exact fill distance is known, it is
enough to increase it by an arbitrarily small positive amount.
\end{remark}

\section{Data-dependent compact multinode Shepard method: construction}\label{sec:construction}

In this section we introduce a nonlinear extension of the compact
multinode Shepard method by incorporating data-dependent weights
inspired by the WENO methodology \cite{Shu}. In \cite{AmatWENO}, the partition of unity structure of B-spline basis functions is exploited to introduce a nonlinear WENO-type
weighting strategy; the same idea is applied to the classical Shepard
method in \cite{LRRY}. Following these ideas, we define, for $x\in\Omega\setminus X$, the
nonlinear operator
\begin{equation}\label{eq:Kdef}
\Kop[f](x)
=\sum_{j\in J_x}\wt B_{\mu,j}[f](x)\,P_j[f](x),
\end{equation}
where
\begin{equation}\label{eq:Bdef}
\wt B_{\mu,j}[f](x)
=\frac{\beta_{\mu,j}[f](x)}
{\displaystyle\sum_{k\in J_x}\beta_{\mu,k}[f](x)},
\qquad
\beta_{\mu,j}[f](x)
=\frac{\wt W_{\mu,j}(x)}{(\epsilon+\Ind{j})^{t}}.
\end{equation}
Here $\Ind{j}\ge0$ is a smoothness indicator associated with the
stencil $\tau_j$, while $t\ge1$ and $\epsilon>0$ are parameters.
The exponent $t$ controls the damping of non-smooth contributions,
whereas $\epsilon$ ensures that the denominator in the definition of
$\beta_{\mu,j}[f]$ remains strictly positive. For the asymptotic analysis, we take
$\epsilon=\epsilon(h)\le h^{d+1}$ and complete the definition of the
operator at the data sites by setting
\[
\Kop[f](x_i)=f_i,\qquad i=1,\dots,n.
\]
Since $P_j[f]$ \emph{interpolates} $f$ on $\tau_j$, the residuals
at the stencil nodes vanish identically and therefore provide no
information on the local regularity. In contrast to the degree-zero
setting of \cite{LRRY}, where the indicator is constructed from an
auxiliary least-squares fit, here the residuals must be evaluated at
nodes outside $\tau_j$. Since these nodes belong to $X$, no
additional function evaluations are required, as the same local
interpolant $P_j[f]$ can be evaluated at them.
\begin{definition}\label{def:indicator}
For $j=1,\dots,r$, let
$\mathcal E_j\subseteq X\setminus\tau_j$ be a nonempty
\emph{indicator set} with $\abs{\mathcal E_j}\le N_E$ and $\dist(y,\tau_j)\le c_Eh$
for every $y\in \mathcal E_j$, where $N_E\in\N$ and $c_E>0$ are constants
independent of $j$ and $h$, $\abs{F}$ denotes the cardinality of a
finite set $F$, and $\dist(y,\tau_j)=\min_{i}\norm{y-x_{j_i}}$.
The smoothness indicator associated with $\tau_j$ is
\begin{equation}\label{eq:indicator}
\Ind{j}=\frac{1}{\abs{\mathcal E_j}}\sum_{x_i\in \mathcal E_j}
\abs{\,f(x_i)-P_j[f](x_i)\,}.
\end{equation}
\end{definition}
In the implementation, where $s=2$, we take $\mathcal E_j=\mathcal N_i\setminus\tau_j$,
where $\mathcal N_i$ is the neighborhood of the seed $x_i$ of $\tau_j$,
from which the stencil is extracted during the stencil selection
(Section~\ref{sec:properties}); this set is nonempty,
since $\abs{\mathcal N_i}\ge\tfrac12(d+2)(d+3)>m$, and it reuses the neighborhood already
constructed
by the search procedure. For the quasi-uniform node sets of
Section~\ref{sec:properties}, this choice satisfies the requirements of
Definition~\ref{def:indicator} with $N_E$ and $c_E$ independent of $h$:
$\mathcal N_i$ is contained in a square of side proportional to $h$ centered at a
node of $\tau_j$, and such a square contains a number of nodes bounded
independently of $h$. This is the choice used in all the numerical
experiments of Section~\ref{sec:numerics}.
The behavior of the indicator is analyzed in
Lemma~\ref{lem:P1P2}. In smooth regions, \eqref{eq:indicator} is
$O(h^{d+1})$; moreover, when the stencil lies on one side of the
discontinuity and its indicator set contains nodes on the other side,
the indicator is uniformly bounded away from zero. Remark~\ref{rem:alias} shows, through
an explicit configuration, that such a lower bound does not extend to
arbitrary stencils crossing the discontinuity.

\section{Properties of the new operator}\label{sec:properties}
Throughout this section $X$ is quasi-uniform and the upper bound $h$
of Section~\ref{sec:multinode} is taken within the quasi-uniformity
window: there exists $c_{qu}>0$, independent of $n$, such that
\[
q_X\;\le\;h_{X,\Omega}\;<\;h\;\le\;c_{qu}\,q_X,
\]
where
\[
q_X:=\tfrac12\min_{i\ne k}\norm{x_i-x_k}
\]
is the separation distance. 
The first inequality uses the connectedness of $\Omega$. Let
$x_a\in X$; since $\Omega$ contains points arbitrarily close to $x_a$
and, for $n\ge2$, points arbitrarily close to a node $x_b$ with
$\norm{x_b-x_a}\ge2q_X$, it meets both $B(x_a,q_X)$ and the complement
of $\overline{B(x_a,q_X)}$; being open and connected, it is not the
union of these two disjoint open sets, hence it contains a point $p$
with $\norm{p-x_a}=q_X$. For every other node $x_c$,
$\norm{p-x_c}\ge\norm{x_c-x_a}-\norm{p-x_a}\ge q_X$, so that
$\dist(p,X)=q_X$ and $h_{X,\Omega}\ge q_X$. The argument also holds for
$X\subset\overline\Omega$. All the estimates below are stated in
terms of $h$; since $h$ and $h_{X,\Omega}$ are comparable, the
corresponding orders are the same, and $h$ is also the quantity used
in the numerical experiments.
For each stencil $\tau_j$, we set
\[
a_j:=\tfrac1m\sum_{i=1}^m x_{j_i},
\]
the barycenter of $\tau_j$.

\begin{assumption}\label{ass:stencils}
There exist constants $c_1,\sigma_0>0$, independent of $h$, such that
for every $j=1,\dots,r$:
\begin{itemize}
\item[\rm(i)] $\diam(\tau_j)\le c_1h$;
\item[\rm(ii)] denoting by $V_j$ the Vandermonde matrix of a fixed
basis of $\Pd$ at the scaled nodes $(x_{j_i}-a_j)/h$, the smallest singular value of $V_j$ satisfies
$\sigma_{\min}(V_j)\ge\sigma_0$.
\end{itemize}
\end{assumption}
Setting $c_2:=c_1+\max\{1,c_E\}$, with $c_E$ the constant of
Definition~\ref{def:indicator}, Assumption~\ref{ass:stencils}(i) and the
definition $R_j=\diam(\tau_j)+h$ give
\[
R_j\le c_2h.
\]
Moreover, every $y\in\overline{C_j}$ (in particular every node of
$\tau_j$)
satisfies
\begin{equation}\label{eq:barycentric}
\norm{y-a_j}
\le\frac1m\sum_{i=1}^m\norm{y-x_{j_i}}
\le R_j
\le c_2h.
\end{equation}
Every point $y$ of an indicator set $\mathcal E_j$ satisfies
$\norm{y-a_j}\le\dist(y,\tau_j)+\diam(\tau_j)\le c_2h$ as well.
Thus all the local objects involved below lie in the closed ball
$\overline{B(a_j,c_2h)}$.

Condition (ii) is a scale-invariant stability requirement. For fixed
$s$, $d$, $c_1$ and $c_E$, the lower bound
$\sigma_{\min}(V_j)\ge\sigma_0$, together with the locality condition
in (i), implies a uniform bound
$\Lambda=\Lambda(s,d,c_1,c_E,\sigma_0)$ on the local Lebesgue constants of
the interpolation operators $f\mapsto P_j[f]$ on the balls
$B(a_j,c_2h)$, by a standard finite-dimensional argument. Condition
(i), in turn, reflects the locality of the neighborhood-based stencil
construction used in the implementation.

In the implementation, where $s=2$, the family $\mathcal S_d$ is built
by a minimal-covering construction, as in \cite{DellAccioSoftware}, in
which each stencil is extracted from a local neighborhood by a discrete
Leja-type selection based on LU factorization with row pivoting
\cite{BosLeja}. The nodes of $X$ are visited in a fixed order, and a new
stencil is built only when a node not yet contained in any stencil is
met, which keeps the number of stencils small (hence the name).
For such a node $x_i$, the \emph{seed}, the construction proceeds as follows.
\begin{enumerate}
\item\label{st:nbh} \emph{Neighborhood.} Let $\mathcal N_i$ be the set of nodes of $X$ in the
square of side $\ell=\bigl(\tfrac12(d+2)(d+3)\operatorname{area}(Q)/n\bigr)^{1/2}$
centered at the seed, where $Q$ is the bounding box of $X$ and
$n=\abs{X}$ is the number of nodes. If this square
contains fewer than $\tfrac12(d+2)(d+3)$ nodes, it is enlarged by the
factors $1+k/10$, $k=1,2,\dots$, until it does. The side $\ell$ is chosen
so that the square contains, on average, $\tfrac12(d+2)(d+3)=\dim\Pi_{d+1}(\R^2)$
nodes, more than the $m=\tfrac12(d+1)(d+2)$ nodes of a stencil: the
selection in the next steps has nodes to choose from, and the indicator
set $\mathcal E_j=\mathcal N_i\setminus\tau_j$ of Section~\ref{sec:construction} is
nonempty.
\item\label{st:leja} \emph{Leja ordering.} The nodes of $\mathcal N_i$ are listed by increasing
distance from the seed, and their Vandermonde matrix is formed in the
monomial basis centered at the barycenter of $\mathcal N_i$, with the monomials in
graded lexicographic order up to the smallest degree $\bar d$ such that
$\dim\Pi_{\bar d}(\R^2)\ge\abs{\mathcal N_i}$. LU factorization with row pivoting,
carried to completion, orders the nodes of $\mathcal N_i$ as a nested discrete
Leja sequence $\xi_1,\xi_2,\dots$ \cite{BosLeja}: at each step the pivot
row is the node that maximizes, in absolute value, the determinant of the
leading block of the matrix. The first column of the matrix is constant,
so that the first pivot is a tie among all the rows, resolved in favour of
the first one: $\xi_1$ is the seed. Since the basis is graded, the first
$m$ elements of the sequence do not depend on $\bar d$.
\item\label{st:cand} \emph{Candidate stencil.} The candidate is $\{\xi_1,\dots,\xi_m\}$,
the first $m$ nodes of the Leja sequence. It contains the seed, and the
greedy choice of the pivots promotes well-conditioned, hence unisolvent,
configurations; unisolvence is in any case guaranteed by the check of the
next step, since $\sigma_{\min}(V_j)\ge\sigma_0>0$.
\item\label{st:check} \emph{Conditioning check and replacement.} The candidate is
accepted if $\sigma_{\min}(V_j)\ge\sigma_0$, with $V_j$ the scaled
Vandermonde matrix of Assumption~\ref{ass:stencils}(ii) and $\sigma_0$ a
prescribed threshold. Otherwise, the node $\xi_{m+1}$ replaces in turn
$\xi_m,\xi_{m-1},\dots,\xi_2$, then $\xi_{m+2}$ does the same, and so on:
each candidate differs from $\{\xi_1,\dots,\xi_m\}$ in a single node, its
barycenter $a_j$ and its matrix $V_j$ are recomputed, and the first
candidate satisfying the bound is accepted. The seed is never replaced.
If the sequence is exhausted without success, the neighborhood is
enlarged by the same factors $1+k/10$, at most ten more times, and steps
\ref{st:leja}--\ref{st:check} are repeated.
\item\label{st:mark} \emph{Marking.} The nodes of the accepted stencil $\tau_j$ are
marked as used, and the visit of $X$ continues with the next unused node.
\end{enumerate}
Two aspects of this procedure deserve a comment. First, the Vandermonde
matrix of step~\ref{st:leja} may be rank deficient: on gridded neighborhoods, for
instance, the polynomial $u^3-u$, with $u$ the first coordinate, vanishes
on $\{-1,0,1\}^2$. A column with vanishing
pivot gives no selection criterion, and the corresponding node is taken
in the order returned by the factorization. This may change the order in
which the candidates of step~\ref{st:check} are tried, and hence the stencil that is
accepted, but not its properties, since the bound
$\sigma_{\min}(V_j)\ge\sigma_0$ is checked directly on every accepted
stencil. Second, the threshold of step~\ref{st:check} is needed: the greedy selection
of step~\ref{st:cand} promotes well-conditioned configurations, but it does not by
itself guarantee (ii) uniformly in $h$. On Halton nodes with $d=1$, for
instance, the smallest value of $\sigma_{\min}(V_j)$ over the candidates
of step~\ref{st:cand} decreases from $0.26$ to $0.045$ between $l=4$ and $l=7$ in the
notation of Section~\ref{sec:numerics}.

The properties required in this section follow from the construction.
Every node is either a seed or a node of a previously built stencil, so
that $\bigcup_j\tau_j=X$; the stencils are pairwise distinct, since each
of them contains a node, its seed, not used before; and (ii) holds with
the prescribed $\sigma_0$ for every accepted stencil. As for (i), for
quasi-uniform $X$ the side $\ell$ is bounded above and below by multiples
of $h$, and a square of side proportional to $h$ centered at a node
contains the required number of nodes, so that the number of enlargements
in steps \ref{st:nbh} and \ref{st:check} is bounded independently of $h$; since the replacing
nodes of step~\ref{st:check} belong to the neighborhood, every stencil lies in a square
of side proportional to $h$, and (i) holds with an $h$-independent
constant $c_1$. The procedure may fail to find a candidate satisfying the
bound of step~\ref{st:check}, in which case the implementation keeps the candidate
with the largest $\sigma_{\min}(V_j)$ and reports the failure; this never
occurred in the experiments of Section~\ref{sec:numerics}, which use
$\sigma_0=0.05$.

No regularity of $\partial\Omega$ enters the proofs below, since
Assumption~\ref{ass:stencils} and the $C^{d+1}$ extensions of the
target functions are required explicitly. Some regularity is, however,
needed for these hypotheses to be satisfiable with $h$-independent
constants: near $\partial\Omega$, well-conditioned stencils of diameter
$O(h)$ exist if $\Omega$ satisfies an interior cone condition, and every
$f\in C^{d+1}(\overline\Omega)$ admits the required extension if, for
instance, $\Omega$ is a bounded Lipschitz domain: such domains are
quasiconvex, i.e.\ they have Whitney's property~(P), and for regions
with this property every function of class $C^{d+1}$ whose derivatives
of order $d+1$ extend continuously to $\partial\Omega$ can be extended
to a $C^{d+1}$ function on $\R^s$ \cite{WhitneyBoundary}, by means of
Whitney's extension theorem \cite{Whitney}.
Both conditions hold for the square used in
Section~\ref{sec:numerics}.
Under Assumption~\ref{ass:stencils}, standard estimates for polynomial
interpolation on unisolvent sets give, for
$g\in C^{d+1}(\overline{B(a_j,c_2h)})$ and every
$y\in\overline{B(a_j,c_2h)}$,
\begin{equation}\label{eq:localerr}
\abs{g(y)-P_j[g](y)}
\le C\Lambda\,h^{d+1}\abs{g}_{d+1,B(a_j,c_2h)},
\qquad
\abs{g}_{d+1,A}
:=\max_{\abs{\alpha}=d+1}\norm{D^\alpha g}_{\infty,A},
\end{equation}
with $C=C(s,d,c_1,c_E)$; here $\norm{g}_{\infty,A}=\sup_A\abs{g}$ for a set
$A$, and the subscript is omitted when the set is clear from the context. Throughout the
paper, $C$ denotes a constant depending only on the parameters indicated
at each occurrence, whose value may change from one occurrence to the
next.

\subsection{Detection properties of the indicator}
We consider a piecewise smooth target. Let $\Gamma\subset\Omega$ be a
$C^1$ separating hypersurface, given as the zero set of
$\gamma\in C^1(\Omega)$ with $\nabla\gamma\ne0$ on a tubular
neighborhood $\mathcal T$ of $\Gamma$ of uniform width, that is, $\mathcal
T\supseteq\{x\in\Omega:\dist(x,\Gamma)<\delta_{\mathcal T}\}$ for some
$\delta_{\mathcal T}>0$; set
$\Omega^+=\{\gamma\ge0\}$, $\Omega^-=\Omega\setminus\Omega^+$
(points of $\Gamma$ belong to $\Omega^+$), and
\begin{equation}\label{eq:ftilde}
\tilde f(x)=
\begin{cases}
f_1(x), & x\in\Omega^+,\\
f_2(x), & x\in\Omega^-.
\end{cases}
\end{equation}
We assume that $f_1$ and $f_2$ admit extensions of class $C^{d+1}$ to
fixed open sets
$U_+\supset\overline{\Omega^+}\cup\mathcal T$ and
$U_-\supset\overline{\Omega^-}\cup\mathcal T$, respectively. The
branch seminorms $\abs{f_1}_{d+1}$ and $\abs{f_2}_{d+1}$ are taken on
$U_+$ and $U_-$, respectively, and are therefore independent of $h$. We also assume that, for some $h_0>0$ and all $h\le h_0$, the
balls $\overline{B(a_j,c_2h)}$ appearing below are contained in the
relevant $U_\pm$, and that the jump satisfies
$\abs{f_1-f_2}\ge J_0>0$ on $\mathcal T$.

\begin{lemma}\label{lem:P1P2}
Let Assumption~\ref{ass:stencils} hold and let
$\tau_j\cup \mathcal E_j\subset B(a_j,c_2h)$, with $B(a_j,c_2h)$ contained in
the domain of the relevant extension.
\begin{itemize}
\item[\rm(P1)] If $f$ coincides on $B(a_j,c_2h)$ with a function of
class $C^{d+1}$, then
\[
\Ind{j}\le C\Lambda h^{d+1}
\abs{f}_{d+1,B(a_j,c_2h)}.
\]
\item[\rm(P2)] If the datum is $\tilde f$ of \eqref{eq:ftilde}, and
$\tau_j\subset\Omega^+$ while $\mathcal E_j$
contains at least one point of $\Omega^-\cap\mathcal T$, or
$\tau_j\subset\Omega^-$ while $\mathcal E_j$ contains at least one point of
$\Omega^+\cap\mathcal T$, then
\[
\Ind{j}\ge \frac{J_0}{2N_E}>0
\]
for all $h$ small enough.
\end{itemize}
\end{lemma}

\begin{proof}
(P1) is immediate from \eqref{eq:localerr}. For (P2), let
$x_i\in \mathcal E_j\cap\Omega^-\cap\mathcal T$. Since
$\tau_j\subset\Omega^+$,
$P_j[\tilde f]=P_j[f_1]$ interpolates the extended smooth branch $f_1$
and, by \eqref{eq:localerr} applied to the extension,
\[
\abs{f_1(x_i)-P_j[\tilde f](x_i)}\le C\Lambda h^{d+1}\abs{f_1}_{d+1}.
\]
Since $x_i\in\mathcal T$, we have
$\abs{f_2(x_i)-f_1(x_i)}\ge J_0$. Hence
\[
\begin{aligned}
\abs{\tilde f(x_i)-P_j[\tilde f](x_i)}
&\ge
\abs{f_2(x_i)-f_1(x_i)}
-\abs{f_1(x_i)-P_j[\tilde f](x_i)}\\
&\ge J_0-C\Lambda h^{d+1}\abs{f_1}_{d+1}
\ge \frac{J_0}{2}
\end{aligned}
\]
for $h$ small enough. The claim then follows from
\[
\Ind{j}\ge
\frac{\abs{\tilde f(x_i)-P_j[\tilde f](x_i)}}{N_E}.
\]
In the opposite case no node of $\tau_j$ lies on $\Gamma\subset\Omega^+$,
so that $P_j[\tilde f]=P_j[f_2]$, and the same argument applies with $f_1$
and $f_2$ exchanged.
\end{proof}
Lemma~\ref{lem:P1P2}(P2) provides a \emph{sufficient} detection
condition: it covers stencils that lie on one side of $\Gamma$ while
their indicator set reaches the opposite side within the tubular
neighborhood $\mathcal T$. If $\Omega$ is convex, the latter requirement is not
restrictive for small $h$: the segment joining a node of $\tau_j$ to a
point of $\mathcal E_j\cap\Omega^-$ lies in $\Omega\cap B(a_j,c_2h)$ and, by
continuity of $\gamma$, meets $\Gamma$, so that every point of
$\mathcal E_j\cap\Omega^-$ lies within $2c_2h$ of $\Gamma$ and hence in
$\mathcal T$ for $h$ small enough. In a nonconvex domain, a point of
$\mathcal E_j\cap\Omega^-$ may instead be close to $\tau_j$ across
$\partial\Omega$ rather than across $\Gamma$, and the restriction to
$\mathcal T$ is then essential.
Lemma~\ref{lem:P1P2}(P2) does \emph{not} cover stencils whose
interpolation nodes themselves straddle $\Gamma$; Remark~\ref{rem:alias} in Section~\ref{subsec:jump}
shows, by an explicit configuration, that this limitation is essential.

\subsection{Structural properties}

\begin{proposition}\label{prop:interp}
For every $f$ and every choice of $\epsilon>0$, $t\ge1$:
\begin{itemize}
\item[\rm(a)] $\wt B_{\mu,j}[f](x)\ge0$ and
\[
\sum_{j\in J_x}\wt B_{\mu,j}[f](x)=1
\]
for every covered $x\notin X$; consequently, $\Kop$ reproduces
$\Pd$ on the covered set.

\item[\rm(b)] The definition of $\Kop[f]$,
given by \eqref{eq:Kdef} for $x\notin X$ and by
$\Kop[f](x_i)=f_i$ at the data sites, is continuous: for every
$x_i\in X$,
\[
\lim_{x\to x_i}\Kop[f](x)=f_i.
\]

\item[\rm(c)] On $\Omega \setminus X$, $\Kop[f]$ has at
least the regularity guaranteed by the compact weights, namely
$C^{\lfloor\mu\rfloor}$ if $\mu\notin\N$ and $C^{\mu-1}$ if
$\mu\in\N$. In particular, even integer values of $\mu$ do not yield
$C^\infty$ regularity across support boundaries.

\item[\rm(d)] For every covered $x\notin X$,
\[
\min_{j\in J_x}P_j[f](x)
\le \Kop[f](x)
\le \max_{j\in J_x}P_j[f](x).
\]
\end{itemize}
\end{proposition}

\begin{proof}
For $x\in\Omega\setminus X$, define the unnormalized compact weights by
\[
\wt W^u_{\mu,j}(x)
:=
\prod_{\nu=1}^{m}
\left(
\frac{1}{\norm{x-x_{j_\nu}}}-\frac{1}{R_j}
\right)_{+}^{\mu},
\qquad j=1,\dots,r,
\]
and set
\[
\omega_j:=(\epsilon+\Ind{j})^{-t}>0.
\]
Substituting the expression of $\wt W_{\mu,j}$ from
\eqref{eq:CMSweights} into \eqref{eq:Bdef}, we obtain, for
$j\in J_x$,
\begin{equation}\label{eq:Bunnormalized}
\wt B_{\mu,j}[f](x)
=
\frac{\omega_j\wt W^u_{\mu,j}(x)}
{\displaystyle\sum_{k=1}^{r}\omega_k\wt W^u_{\mu,k}(x)}.
\end{equation}
Indeed, $\wt W^u_{\mu,k}(x)=0$ whenever $k\notin J_x$.

\begin{itemize}
\item[\rm(a)]
Nonnegativity and normalization follow immediately from
\eqref{eq:Bunnormalized}. If $p\in\Pd$, then $P_j[p]=p$ by
unisolvence, and hence, for $x\notin X$,
\[
\Kop[p](x)
=
p(x)\sum_{j\in J_x}\wt B_{\mu,j}[p](x)
=
p(x).
\]
At the data sites, polynomial reproduction follows directly from the
definition of $\Kop$.

\item[\rm(b)]
Fix $x_i\in X$ and set
\[
\rho:=\norm{x-x_i},
\]
and let $J_i=\{j:x_i\in\tau_j\}$ be the set introduced in
Proposition~\ref{prop:W}.
The set $J_i$ is nonempty since the stencil family covers $X$.
Let $j\in J_i$ and let $x_{j_q}=x_i$. Then
\[
\rho^\mu\wt W^u_{\mu,j}(x)
\longrightarrow
w_{ij}
:=
\prod_{\substack{\nu=1\\ \nu\ne q}}^{m}
\left(
\frac{1}{\norm{x_i-x_{j_\nu}}}
-\frac{1}{R_j}
\right)^{\mu}
>0
\]
as $x\to x_i$. Indeed,
\[
\norm{x_i-x_{j_\nu}}
\le \diam(\tau_j)<R_j,
\qquad \nu\ne q,
\]
so all the remaining factors have positive finite limits.

If $j\notin J_i$, no factor of $\wt W^u_{\mu,j}$ is singular at
$x_i$, and therefore
\[
\wt W^u_{\mu,j}(x)=O(1),
\qquad
\rho^\mu\wt W^u_{\mu,j}(x)\longrightarrow0.
\]
Consequently,
\[
\rho^\mu
\sum_{k=1}^{r}\omega_k\wt W^u_{\mu,k}(x)
\longrightarrow
\sum_{k\in J_i}\omega_k w_{ik}>0.
\]
It follows from \eqref{eq:Bunnormalized} that
\[
\wt B_{\mu,j}[f](x)
\longrightarrow
\begin{cases}
\displaystyle
\frac{\omega_j w_{ij}}
{\sum_{k\in J_i}\omega_k w_{ik}},
& j\in J_i,\\[3mm]
0,
& j\notin J_i.
\end{cases}
\]
The nonzero limits are nonnegative and sum to one. Since
$P_j[f](x_i)=f_i$ for every $j\in J_i$, we conclude that
\[
\lim_{x\to x_i}\Kop[f](x)
=
\sum_{j\in J_i}
\frac{\omega_j w_{ij}}
{\sum_{k\in J_i}\omega_k w_{ik}}
P_j[f](x_i)
=
f_i.
\]

\item[\rm(c)]
For fixed $f$, the coefficients $\omega_j$ in
\eqref{eq:Bunnormalized} are positive constants with respect to $x$.
Away from the data sites, a factor of $\wt W^u_{\mu,j}$ may lose
smoothness only across a support boundary
\[
\norm{x-x_{j_\nu}}=R_j.
\]
Writing $\varrho(x)=\norm{x-x_{j_\nu}}$, near such a boundary one has
\[
\left(
\frac{1}{\varrho(x)}-\frac{1}{R_j}
\right)_{+}^{\mu}
=
\frac{(R_j-\varrho(x))_{+}^{\mu}}
{(\varrho(x)R_j)^{\mu}}.
\]
The denominator is smooth and strictly positive there, so the
regularity is governed by $(R_j-\varrho(x))_{+}^{\mu}$. Hence each factor
is of class $C^{\lfloor\mu\rfloor}$ if $\mu\notin\N$ and
$C^{\mu-1}$ if $\mu\in\N$.

Products and finite sums preserve at least this regularity. Moreover,
the denominator in \eqref{eq:Bunnormalized} is strictly positive on
$\Omega\setminus X$ by the covering property. Therefore the nonlinear
weights, and hence $\Kop[f]$, have the stated regularity on
$\Omega\setminus X$, since the local approximants $P_j[f]$ are
polynomials. At the data sites, part (b) guarantees continuity;
higher regularity there is not addressed here.

\item[\rm(d)]
By part (a), $\Kop[f](x)$ is a convex combination of the values
$P_j[f](x)$. Therefore,
\[
\min_{j\in J_x}P_j[f](x)
\le \Kop[f](x)
\le \max_{j\in J_x}P_j[f](x).
\]
\end{itemize}
\end{proof}

\subsection{Convergence in smooth regions}

\begin{theorem}[order in smooth regions]\label{th:smooth}
Let Assumption~\ref{ass:stencils} hold and let $f$ admit a
$C^{d+1}$ extension to a fixed open set containing
\[
\bigcup_j \overline{B(a_j,c_2h)}
\]
for all sufficiently small $h$. Then, for any $\epsilon>0$ and
$t\ge1$,
\[
\norm{f-\Kop[f]}_{\infty,\Omega}
\le
C\,h^{d+1}\,\abs{f}_{d+1},
\]
where $C=C(s,d,c_1,c_E,\sigma_0)$ is independent of $h$, $\epsilon$ and
$t$, and $\abs{f}_{d+1}$ denotes the seminorm of the extension.
\end{theorem}

\begin{proof}
Fix $x\in\Omega\setminus X$. By Lemma~\ref{lem:coverage},
$J_x\neq\emptyset$. For every $j\in J_x$, one has
$x\in C_j$, and hence \eqref{eq:barycentric} gives
\[
\norm{x-a_j}\le c_2h.
\]
Therefore $x\in\overline{B(a_j,c_2h)}$, and
\eqref{eq:localerr} yields
\[
\abs{f(x)-P_j[f](x)}
\le
Ch^{d+1}\abs{f}_{d+1}.
\]
Using Proposition~\ref{prop:interp}(a),
\[
\begin{aligned}
\abs{f(x)-\Kop[f](x)}
&=
\left|
\sum_{j\in J_x}
\wt B_{\mu,j}[f](x)
\bigl(f(x)-P_j[f](x)\bigr)
\right|\\
&\le
\sum_{j\in J_x}
\wt B_{\mu,j}[f](x)
\abs{f(x)-P_j[f](x)}\\
&\le
Ch^{d+1}\abs{f}_{d+1}.
\end{aligned}
\]
At the data sites the estimate is immediate from interpolation.
Taking the supremum over $\Omega$ proves the result.

No additional restriction on $t$ or $\epsilon$ is needed: in a smooth
region all active local interpolants already achieve the same
$O(h^{d+1})$ accuracy, and any convex combination of them preserves
this order.
\end{proof}

\subsection{Behaviour near a discontinuity}\label{subsec:jump}
Let $\tilde f$ be as in \eqref{eq:ftilde} and fix an evaluation point
$x\in\Omega^{\pm}\setminus X$. By Lemma~\ref{lem:coverage},
$J_x\neq\emptyset$. We call an active stencil
$j\in J_x$ \emph{one-sided} if
\[
\tau_j\subset\Omega^+
\qquad\text{or}\qquad
\tau_j\subset\Omega^-,
\]
and \emph{straddling} otherwise.

\begin{assumption}[indicator separation at $x$]\label{ass:separation}
There exist $\theta,\gamma_0\in(0,1)$ and $C_G,c_B>0$, independent of
$h$, such that:
\begin{itemize}
\item[\rm(A1)] if $x\in\Omega^+$ (respectively, $x\in\Omega^-$),
then every one-sided active stencil satisfies
$\tau_j\subset\Omega^+$ (respectively, $\tau_j\subset\Omega^-$);
\item[\rm(A2)] there exists a one-sided active stencil
$j^*\in J_x$ such that
\[
\norm{x-x_{j^*_i}}
\le (1-\theta)R_{j^*},
\qquad i=1,\dots,m,
\]
and
\[
\Ind{j^*}\le C_Gh^{d+1};
\]
\item[\rm(A3)] every active stencil $j\in J_x$ that is not
one-sided satisfies
\[
\Ind{j}\ge c_B;
\]
\item[\rm(A4)] every active stencil $j\in J_x$ that is not
one-sided satisfies
\[
\norm{x-x_{j_i}}\ge \gamma_0 q_X,
\qquad i=1,\dots,m.
\]
\end{itemize}
\end{assumption}
The four conditions have the following meaning.
Condition (A1) excludes one-sided active stencils belonging to the
opposite branch. Such a stencil interpolates smooth data and has a small
indicator, exactly like the one-sided stencils of the correct branch, so
that the nonlinear weighting cannot tell the two apart: if it is active,
it contributes a value of the wrong branch with a weight of order one,
and the error at $x$ is of the order of the jump. This is what happens
at the few points near the corner of the square geometry of
Section~\ref{sec:numerics} where the operator loses accuracy.

Condition (A2) requires a \emph{reference} stencil: one-sided, on the
correct side, smooth in the sense of the indicator, and with $x$ at a
positive distance from the boundary of its support. The geometric margin
plays the same role as the lower-bound condition imposed on the radial
weight in \cite[Theorem~3]{LRRY}: it prevents the unnormalized weight of
the reference stencil from becoming arbitrarily small near the boundary
of its support, and it is with this weight that the weights of the
straddling stencils are compared in the proof. The bound
\[
\Ind{j^*}\le C_Gh^{d+1}
\]
holds, by Lemma~\ref{lem:P1P2}(P1) applied to the extended branch,
whenever the indicator set $\mathcal E_{j^*}$ lies on the same side of $\Gamma$
as $\tau_{j^*}$. Without a
reference stencil, all the active stencils may straddle $\Gamma$, and no
choice of the weights can produce an accurate value: this is what
happens near the points where the line geometry of
Section~\ref{sec:numerics} meets $\partial\Omega$.

The active stencils subject to (A3) are precisely the straddling ones;
under (A1), all the other active stencils lie on the side of $x$, so
that (A3) is the indicator-separation hypothesis proper:
it asks that the indicator detects every active straddling stencil, with
a lower bound $c_B$ independent of $h$, so that the nonlinear weighting
damps its contribution by a factor of order $(h^{d+1}/c_B)^t$.
Lemma~\ref{lem:P1P2}(P2) does not cover this case, since it concerns
one-sided stencils whose indicator set crosses $\Gamma$, and the
following configuration shows that the lower bound in (A3) cannot be
deduced from Definition~\ref{def:indicator} alone.
\begin{remark}[an undetected straddling stencil]\label{rem:alias}
Let $s=2$, $d=1$, and write $(u,v)$ for the coordinates in $\R^2$. Let
$\Gamma=\{u=0\}$ and let $\tilde f$ be the unit step,
\[
\tilde f(u,v)=
\begin{cases}
0, & u<0,\\
1, & u\ge0.
\end{cases}
\]
Consider
\[
\tau_j=\{(-h,0),\,(h,-h),\,(h,h)\},
\qquad
\mathcal E_j=\{(-h,h)\}.
\]
The stencil is unisolvent, satisfies
Assumption~\ref{ass:stencils} uniformly in $h$, and straddles
$\Gamma$. Its linear interpolant is
\[
P_j(u,v)=\frac{u+h}{2h}.
\]
Since $P_j(-h,h)=0=\tilde f(-h,h)$, one has
\[
\Ind{j}=0,
\]
although the stencil crosses a jump of size one. Hence, the residual
indicator of Definition~\ref{def:indicator} does not, in general,
provide a uniform positive lower bound for arbitrary straddling
stencils, and the lower bound required in (A3) must be imposed
separately.
\end{remark}
Finally, (A4) keeps $x$ uniformly separated, at the scale $q_X$, from
the nodes of the straddling stencils. The compact multinode weights are
singular at the nodes, and a straddling stencil with a node very close
to $x$ would have an unnormalized weight large enough to compensate the
damping produced by (A3); (A4) bounds this weight in terms of $q_X$, and
it is used only in the weight comparison of the proof.
\begin{theorem}[behaviour near the discontinuity]\label{th:jump}
Let Assumptions~\ref{ass:stencils} and~\ref{ass:separation} hold,
let $0<\epsilon\le h^{d+1}$, and let $t\ge1$ be fixed. Then, for all
sufficiently small $h$,
\[
\abs{\tilde f(x)-\Kop[\tilde f](x)}
\le
K_1h^{d+1}+K_2h^{(d+1)t},
\]
where
\[
K_1=
C(s,d,c_1,c_E,\sigma_0)
\max\{\abs{f_1}_{d+1},\abs{f_2}_{d+1}\},
\]
and
\[
K_2=
N_{\mathrm{act}}
\Bigl(\frac{c_2c_{qu}}{\gamma_0\theta}\Bigr)^{\mu m}
\Bigl(\frac{1+C_G}{c_B}\Bigr)^t
(1+\Lambda)\norm{\tilde f}_{\infty}.
\]
Here $N_{\mathrm{act}}$ is an $h$-independent upper bound on
$\abs{J_x}$, established in the proof. In particular, for every
fixed $t\ge1$,
\[
\abs{\tilde f(x)-\Kop[\tilde f](x)}
=O(h^{d+1}).
\]
The constants are independent of $h$ and $\epsilon$; their dependence
on the remaining parameters is displayed above, together with the
dependence on the $C^{d+1}$ seminorms of the extended branches.
\end{theorem}
\begin{proof}
Let $\mathcal G$ denote the set of one-sided active stencils and set
\[
\mathcal B:=J_x\setminus\mathcal G.
\]
By (A1), all stencils in $\mathcal G$ lie on the same side of
$\Gamma$ as $x$, while every stencil in $\mathcal B$ is not
one-sided. Without loss of generality, let $x\in\Omega^+$.

We first show that the number of active stencils is bounded
independently of $h$.
If $j\in J_x$, then $x\in\intr(C_j)$ and therefore, by the
definition of $C_j$,
\[
\norm{x-x_{j_i}}<R_j\le c_2h,
\qquad i=1,\dots,m.
\]
Hence all the nodes of every active stencil belong to
$B(x,c_2h)$. The open balls $B(x_i,q_X)$ centered at the data sites
in $X\cap B(x,c_2h)$ are pairwise disjoint and are contained in
$B(x,c_2h+q_X)$. Comparing volumes and using
$h\le c_{qu}q_X$ gives
\[
\abs{X\cap B(x,c_2h)}
\le
\Bigl(\frac{c_2h+q_X}{q_X}\Bigr)^s
\le
(1+c_2c_{qu})^s
=:\kappa.
\]
Since $\mathcal S_d$ consists of pairwise distinct $m$-subsets of
$X$, it follows that
\[
\abs{J_x}
\le
\binom{\lfloor\kappa\rfloor}{m}
=:N_{\mathrm{act}},
\]
which is independent of $h$. Distinctness is essential here:
allowing repeated copies of a straddling stencil would let their total
normalized weight remain $O(1)$, however strong the damping of each
copy. In practice, the minimal-covering construction keeps
$\abs{J_x}$ much smaller than this crude bound.

We next estimate the contribution of the one-sided stencils.
For $j\in\mathcal G$, one has
$\tau_j\subset\Omega^+$ and therefore
\[
P_j[\tilde f]=P_j[f_1].
\]
Moreover, since $j$ is active, \eqref{eq:barycentric} gives
$x\in\overline{B(a_j,c_2h)}$. Hence, by
\eqref{eq:localerr},
\[
\abs{\tilde f(x)-P_j[\tilde f](x)}
\le
Ch^{d+1}\abs{f_1}_{d+1},
\qquad j\in\mathcal G.
\]

We now bound the normalized weights of the stencils in $\mathcal B$.
By (A2), the reference stencil $j^*$ satisfies
\[
\norm{x-x_{j^*_i}}\le(1-\theta)R_{j^*},
\qquad i=1,\dots,m,
\]
and therefore $\beta_{\mu,j^*}[\tilde f](x)>0$. Thus, for
$j\in\mathcal B$,
\[
\begin{aligned}
\wt B_{\mu,j}[\tilde f](x)
&\le
\frac{\beta_{\mu,j}[\tilde f](x)}
     {\beta_{\mu,j^*}[\tilde f](x)}\\
&=
\frac{\wt W^u_{\mu,j}(x)}
     {\wt W^u_{\mu,j^*}(x)}
\left(
\frac{\epsilon+\Ind{j^*}}
     {\epsilon+\Ind{j}}
\right)^t,
\end{aligned}
\]
where $\wt W^u_{\mu,j}$ denotes the unnormalized numerator in
\eqref{eq:CMSweights}.

For $j\in\mathcal B$, each factor satisfies
\[
\left(
\frac{1}{\norm{x-x_{j_i}}}-\frac{1}{R_j}
\right)_+^\mu
\le
\norm{x-x_{j_i}}^{-\mu}.
\]
By (A4) and $h\le c_{qu}q_X$,
\[
\norm{x-x_{j_i}}^{-\mu}
\le
(\gamma_0q_X)^{-\mu}
\le
\left(\frac{c_{qu}}{\gamma_0}\right)^\mu h^{-\mu},
\]
and consequently
\[
\wt W^u_{\mu,j}(x)
\le
\left(\frac{c_{qu}}{\gamma_0}\right)^{\mu m}
h^{-\mu m}.
\]

On the other hand, setting
\[
\eta_i:=\norm{x-x_{j^*_i}},
\]
condition (A2) gives
\[
\eta_i\le(1-\theta)R_{j^*},
\]
and hence
\[
\frac1{\eta_i}-\frac1{R_{j^*}}
=
\frac{R_{j^*}-\eta_i}{\eta_iR_{j^*}}
\ge
\frac{\theta}{R_{j^*}}
\ge
\frac{\theta}{c_2h}.
\]
Therefore,
\[
\wt W^u_{\mu,j^*}(x)
\ge
\left(\frac{\theta}{c_2}\right)^{\mu m}
h^{-\mu m},
\]
and thus
\[
\frac{\wt W^u_{\mu,j}(x)}
     {\wt W^u_{\mu,j^*}(x)}
\le
\Bigl(\frac{c_2c_{qu}}
{\gamma_0\theta}\Bigr)^{\mu m}.
\]

Finally, using
$\epsilon\le h^{d+1}$,
$\Ind{j^*}\le C_Gh^{d+1}$ from (A2), and
$\Ind{j}\ge c_B$ from (A3), we obtain
\[
\wt B_{\mu,j}[\tilde f](x)
\le
\Bigl(\frac{c_2c_{qu}}
{\gamma_0\theta}\Bigr)^{\mu m}
\Bigl(\frac{1+C_G}{c_B}\Bigr)^t
h^{(d+1)t},
\qquad j\in\mathcal B.
\]

The estimate follows by combining the three bounds.
Using the partition of unity,
\[
\begin{aligned}
\abs{\tilde f(x)-\Kop[\tilde f](x)}
&\le
\sum_{j\in\mathcal G}
\wt B_{\mu,j}[\tilde f](x)
\abs{\tilde f(x)-P_j[\tilde f](x)}
\\
&\quad+
\sum_{j\in\mathcal B}
\wt B_{\mu,j}[\tilde f](x)
\abs{\tilde f(x)-P_j[\tilde f](x)}.
\end{aligned}
\]
By the estimate on the one-sided stencils and the normalization of the
nonlinear weights, the first sum is bounded by
\[
Ch^{d+1}\abs{f_1}_{d+1}.
\]
Furthermore, the uniform bound on the local Lebesgue constants gives
\[
\abs{P_j[\tilde f](x)}
\le
\Lambda\norm{\tilde f}_{\infty},
\]
and hence
\[
\abs{\tilde f(x)-P_j[\tilde f](x)}
\le
(1+\Lambda)\norm{\tilde f}_{\infty},
\qquad j\in\mathcal B.
\]
Using the bound on $\abs{J_x}$ and the bound on the weights of the
stencils in $\mathcal B$, the second sum is therefore bounded by
\[
N_{\mathrm{act}}
\Bigl(\frac{c_2c_{qu}}
{\gamma_0\theta}\Bigr)^{\mu m}
\Bigl(\frac{1+C_G}{c_B}\Bigr)^t
(1+\Lambda)\norm{\tilde f}_{\infty}
h^{(d+1)t}.
\]
This proves the stated estimate. The case
$x\in\Omega^-$ is analogous.
\end{proof}

\begin{remark}[scope of the theorem]\label{rem:scope}
Theorem~\ref{th:jump} is a pointwise conditional result: the
$O(h^{d+1})$ rate holds at every evaluation point $x$ satisfying
Assumption~\ref{ass:separation}. No uniform validity of
(A1)--(A4) with respect to $x$ is assumed here. Obtaining the
rate at every point at distance larger than a fixed multiple of $h$
from $\Gamma$ would require a covering-regularity property of the
selected stencil family, namely the existence, near every such point,
of a one-sided reference stencil as in (A2), together with (A1), (A3) and
(A4). Establishing such a property for the minimal-covering
construction used in the implementation is an open problem; the
experiments of Section~\ref{sec:numerics} are consistent with it for Halton
nodes and the circular discontinuity. Denoting by $H$ the computable
upper bound of $h_{X,\Omega}$ used in the experiments, see \eqref{eq:Hbound}
in Section~\ref{sec:numerics}, at the four levels of
Table~\ref{tab:smooth} the error of the nonlinear operator at distance
larger than $2H$ from $\Gamma$ is at most $1.2$ times its maximum error in
the smooth case, and the band where its error exceeds a quarter of the jump
has width between $0.6H$ and $0.8H$\ (Table~\ref{tab:bandref}). The
replacement of
ill-conditioned stencils described after Assumption~\ref{ass:stencils}
also changes the nodes removed during the covering construction, and
hence the subsequent stencils: the covering near $\Gamma$ is sensitive
to local choices, which is a further reason why its regularity cannot
be taken for granted. We note that, if
all the active stencils at a given point $x$ are one-sided and
belong to the same branch as $x$, the argument of
Theorem~\ref{th:smooth} applies directly and conditions (A3)--(A4)
are not needed.
\end{remark}

\medskip\noindent\textit{Choice of the exponent $t$.}
Under the condition $0<\epsilon\le h^{d+1}$, any fixed $t\ge1$
gives the order $O(h^{d+1})$ in Theorem~\ref{th:jump}. The estimate,
however, does not provide an optimal choice of $t$ at finite
resolution. Indeed, the constant $K_2$ contains the factor
\[
\Bigl(\frac{c_2c_{qu}}{\gamma_0\theta}\Bigr)^{\mu m},
\]
whereas the contribution of the detected stencils contains
\[
\Bigl(\frac{1+C_G}{c_B}\Bigr)^t h^{(d+1)t}.
\]
Since the constants involved are not known a priori, these estimates
give only qualitative indications on the choice of $t$. The
second factor decays faster as $d$ grows, while the first one grows
with $m=\binom{d+s}{s}$; at finite resolution this suggests that larger
values of $t$ may be needed at higher degrees to compensate for the
larger geometric prefactor. In the
numerical experiments we use $t=4,6,8,12$ for
$d=1,2,3,4$, respectively. The sensitivity study of
Section~\ref{sec:numerics} shows that, among the values tested,
$t=6$ gives the smallest error for $d=2$, while for $d=4$ the error
decreases up to $t=12$\ and remains
essentially unchanged for $t=14$ and $t=16$.
\medskip

\begin{remark}[comparison with the degree-zero case]
For $d=0$, the local interpolants are constants and the construction
reduces to a compactly supported WENO--Shepard-type method.
Theorems~\ref{th:smooth} and~\ref{th:jump} then give a first-order
approximation rate, as in \cite{LRRY}. The two methods use different
smoothness indicators: in \cite{LRRY} the indicator is obtained from
an auxiliary least-squares fit, whereas here it is defined through the
residuals in Definition~\ref{def:indicator}.

For $d\ge1$, the use of local polynomial interpolants gives the order
$O(h^{d+1})$ in smooth regions, and Theorem~\ref{th:jump} shows that
the same order is preserved at points satisfying
Assumption~\ref{ass:separation}. In the experiments of
Section~\ref{sec:numerics} with the circular discontinuity, at the four levels of
Table~\ref{tab:smooth} and for every degree, the error of the nonlinear
operator at distance larger than $2H$ from $\Gamma$ does not exceed $1.2$
times its maximum error in the smooth case: high order and sharp
resolution of the discontinuity are obtained simultaneously. The linear and nonlinear compact
multinode operators are both convex combinations of the active local interpolants; the difference is that the nonlinear weights reduce the
contribution of the stencils detected as non-smooth.
\end{remark}

\section{Numerical experiments}\label{sec:numerics}
All the experiments are performed on $\Omega=(0,1)^2$, with $N$ nodes
($N=n$ in the notation of Section~\ref{sec:shepard}), $\mu=4$ and
$\epsilon=10^{-14}$; points of $\R^2$ are written as $(x,y)$ in this
section. Note that $\epsilon\le h^{d+1}$ holds throughout
the reported resolutions: since $N$ balls of radius $h$ cover
$[0,1]^2$, one has $h\ge(N\pi)^{-1/2}$, which at the finest level
($N=16641$) gives $h>4.37\cdot10^{-3}$ and
$h^{5}\ge(N\pi)^{-5/2}>1.6\cdot10^{-12}>10^{-14}$, consistently with the theory. 
Two node distributions are used: Halton points and regular grids
with the same number $N_l=(2^l+1)^2$ of nodes. The grid nodes on the
boundary of $[0,1]^2$ do not belong to $\Omega$; the definitions and
results of Sections~\ref{sec:shepard}--\ref{sec:properties} remain
valid, with the same proofs, for $X\subset\overline\Omega$, which is the
setting used for the grids, provided that $f$ and, in
Section~\ref{sec:properties}, $\gamma$, the sets $\Omega^\pm$ and the
branches $f_1,f_2$ are understood as defined on $\overline\Omega$, so that
the data at the boundary nodes and their assignment to a branch are well
defined.
The stencils are
selected by the discrete Leja-type procedure described in
Section~\ref{sec:properties} (LU with row
pivoting on the local Vandermonde matrix in a translated and scaled
monomial basis \cite{BosLeja});  stencils violating
$\sigma_{\min}(V_j)\ge\sigma_0=0.05$ are replaced along the nested Leja
sequence as described in Section~\ref{sec:properties}. On Halton nodes
this happens only at the two finest levels (at most $54$ stencils, for
$d=4$ and $l=7$), on regular grids, where nearly collinear
configurations are more frequent, only for $d\ge3$ (at most $164$
stencils, for $d=4$ and $l=7$), and no stencil remained below
$\sigma_0$ in any experiment. In accordance with Section~\ref{sec:multinode}, the
parameter $h$ is the certified computable upper bound
\begin{equation}\label{eq:Hbound}
h_{X,\Omega}\;\le\;H\;:=\;\max_{z\in G_\delta}\min_{x_i\in X}
\norm{z-x_i}\;+\;\delta_G,
\end{equation}
where $G_\delta$ is an auxiliary uniform grid of covering radius
$\delta_G$ (in the implementation, $\delta_G=\delta/\sqrt2$ with
$\delta=1/(4\sqrt N)$ the grid spacing, and the grid covers the convex hull of
$X$ enlarged by $\delta_G$); since this set does not contain the
whole of $\Omega$, we also evaluated the same bound on a finer grid, of spacing $\delta/4$,
covering
$[0,1]^2$, obtaining $0.0781$, $0.0407$ and $0.0226$ for $N=289$,
$1089$ and $4225$, all smaller than the corresponding values of $H$ in
Table~\ref{tab:smooth} (for $N=16641$ the bound is $0.0119$,
against $H=0.0129$); on the regular grids $H$ exceeds the exact fill
distance $2^{-l}/\sqrt2$ at every level, so that $H>h_{X,\Omega}$ holds in every
experiment. The strict
inequality required by Lemma~\ref{lem:coverage} is therefore certified
at every level,  and no
uncovered evaluation point occurred in any experiment. The weights are evaluated in logarithmic
form with an incremental log-sum-exp normalization, stable up to the
data sites, where the interpolation value is assigned directly. The
smoothness indicators are computed on the sets $\mathcal E_j=\mathcal N_i\setminus\tau_j$
introduced after
Definition~\ref{def:indicator}, and the exponents are paired with the
degree as $d=1,2,3,4\leftrightarrow t=4,6,8,12$
(Section~\ref{subsec:jump}). In tables and captions the
data-dependent (nonlinear) CMS is labelled WENO, and the CMS of
Section~\ref{sec:multinode} linear. As reference smooth datum we use Franke's
function
\begin{equation}\label{eq:franke}
\begin{aligned}
f(x,y)={}&\tfrac34e^{-\frac14((9x-2)^2+(9y-2)^2)}
+\tfrac34e^{-\frac1{49}(9x+1)^2-\frac1{10}(9y+1)}\\
&+\tfrac12e^{-\frac14((9x-7)^2+(9y-3)^2)}
-\tfrac15e^{-(9x-4)^2-(9y-7)^2}.
\end{aligned}
\end{equation}
The errors of an approximant $\mathcal I$ ($\mathcal I=\wt M_\mu$ for the
linear operator, $\mathcal I=\Kop$ for the nonlinear one) are measured on
a uniform evaluation grid of $N_{\rm eval}=100^2$ points $z_i$ as the
maximum error and the root mean square error
\[
\Einf=\max_i\abs{f(z_i)-\mathcal I(z_i)},\qquad
\mathrm{RMSE}=\Bigl(\tfrac1{N_{\rm eval}}\textstyle\sum_i
(f(z_i)-\mathcal I(z_i))^2\Bigr)^{1/2},
\]
with rates $p^\infty_l=\log(\Einf^{(l-1)}/\Einf^{(l)})/
\log(H_{l-1}/H_l)$, where $\Einf^{(l)}$ and $H_l$ are the values of
$\Einf$ and $H$ at level $l$, and analogously $p^2_l$, computed with
unrounded values of $H$. (The quantity $\Einf$ is called maximum
absolute error, MAE, in \cite{LRRY}; we prefer the unambiguous
notation.)

\subsection{Order of accuracy in smooth regions}

We approximate \eqref{eq:franke} from $N_l=(2^l+1)^2$ Halton nodes and
regular grid nodes,
$l=4,5,6,7$, for local degrees $d=1,\dots,4$. Table~\ref{tab:smooth}
reports the results for the Halton nodes, Table~\ref{tab:smoothgrid}
those for the regular grids. The linear and the nonlinear operators exhibit the
same convergence orders at every degree, in accordance with
Theorem~\ref{th:smooth}. The root mean square errors are consistent with the order $d+1$ (rates
$2.02$, $3.12$, $3.99$ and $5.20$ at the finest step for the nonlinear
operator, after values above $d+1$ at the previous step; least squares
slopes over the four levels $2.05$, $3.09$,
$3.95$ and $4.95$), whereas the maximum-norm rates oscillate between
consecutive levels, e.g.\ $1.12$, $2.72$, $1.45$ for $d=1$. On scattered
nodes the stencils, and hence the local constants in \eqref{eq:localerr},
change from one level to the next; the oscillation of the maximum-norm rates
is consistent with this variation rather than with a loss of order, as the
least squares slopes of the root mean square errors and the results on the
regular grids indicate. On the regular grids the maximum-norm rates of the
linear
operator do not show the oscillations observed on the Halton set: they stay
between $2.02$ and $2.04$ for $d=1$ and increase monotonically for $d\ge2$,
and at the finest step
they are $2.04$, $3.10$, $3.93$ and $5.24$; those of
the nonlinear operator show small oscillations for $d=3$; at
equal number of nodes the maximum errors on the grids are smaller than on
the Halton
set for $d=1$ and $d=4$ at every level, but not for $d=2$ and $d=3$: the
grid error is the larger one at $d=2$, $l=4,6$ and $d=3$, $l=5,6$ for the
linear operator, and at $d=2$, $l=4,5,6$ and at every level for $d=3$ for
the nonlinear one.
At the finest level, on the Halton set and for the linear
operator, the maximum error decreases by nearly 
three orders of
magnitude (a factor $\approx$\,$710$) from $d=1$ to $d=4$ at equal number of
nodes.

\begin{table}[t]
\centering\small
\caption{Smooth Franke data, Halton nodes: errors and rates for the
linear CMS and the data-dependent (WENO) CMS, local degrees
$d=1,\dots,4$ with exponents $t=4,6,8,12$.}
\label{tab:smooth}
\begin{tabular}{cccc cc cccc}
\toprule
$d$ & $l$ & $N$ & $H$ &
$\Einf^{\rm lin}$ & $p^\infty$ &
$\Einf^{\rm WENO}$ & $p^\infty$ &
$\mathrm{RMSE}^{\rm WENO}$ & $p^2$\\
\midrule
1 & 4 & 289   & 0.0879 & 6.168e-02 & --   & 5.568e-02 & --   & 9.343e-03 & --\\
  & 5 & 1089  & 0.0421 & 2.709e-02 & 1.12 & 2.702e-02 & 0.98 & 2.937e-03 & 1.57\\
  & 6 & 4225  & 0.0246 & 6.255e-03 & 2.72 & 6.174e-03 & 2.74 & 7.152e-04 & 2.62\\
  & 7 & 16641 & 0.0129 & 2.469e-03 & 1.45 & 2.810e-03 & 1.23 & 1.964e-04 & 2.02\\
\midrule
2 & 4 & 289   & 0.0879 & 3.749e-02 & --   & 3.922e-02 & --   & 4.652e-03 & --\\
  & 5 & 1089  & 0.0421 & 1.153e-02 & 1.60 & 7.933e-03 & 2.17 & 7.282e-04 & 2.52\\
  & 6 & 4225  & 0.0246 & 1.028e-03 & 4.48 & 8.636e-04 & 4.11 & 9.947e-05 & 3.69\\
  & 7 & 16641 & 0.0129 & 3.563e-04 & 1.65 & 1.578e-04 & 2.65 & 1.344e-05 & 3.12\\
\midrule
3 & 4 & 289   & 0.0879 & 3.692e-02 & --   & 2.489e-02 & --   & 2.452e-03 & --\\
  & 5 & 1089  & 0.0421 & 5.135e-03 & 2.68 & 3.783e-03 & 2.56 & 2.925e-04 & 2.89\\
  & 6 & 4225  & 0.0246 & 2.801e-04 & 5.40 & 2.352e-04 & 5.15 & 1.873e-05 & 5.10\\
  & 7 & 16641 & 0.0129 & 3.280e-05 & 3.35 & 2.455e-05 & 3.53 & 1.453e-06 & 3.99\\
\midrule
4 & 4 & 289   & 0.0879 & 4.490e-02 & --   & 4.274e-02 & --   & 1.747e-03 & --\\
  & 5 & 1089  & 0.0421 & 2.055e-03 & 4.19 & 2.056e-03 & 4.13 & 9.419e-05 & 3.97\\
  & 6 & 4225  & 0.0246 & 8.994e-05 & 5.81 & 9.274e-05 & 5.75 & 4.092e-06 & 5.82\\
  & 7 & 16641 & 0.0129 & 3.477e-06 & 5.08 & 2.678e-06 & 5.53 & 1.458e-07 & 5.20\\
\bottomrule
\end{tabular}
\end{table}

\begin{table}[t]
\centering\small
\caption{Smooth Franke data, regular grid nodes: errors and rates
for the linear CMS and the data-dependent (WENO) CMS, local degrees
$d=1,\dots,4$ with exponents $t=4,6,8,12$.}
\label{tab:smoothgrid}
\begin{tabular}{cccc cc cccc}
\toprule
$d$ & $l$ & $N$ & $H$ &
$\Einf^{\rm lin}$ & $p^\infty$ &
$\Einf^{\rm WENO}$ & $p^\infty$ &
$\mathrm{RMSE}^{\rm WENO}$ & $p^2$\\
\midrule
1 & 4 & 289 & 0.0520 & 2.865e-02 & -- & 2.864e-02 & -- & 5.144e-03 & --\\
  & 5 & 1089 & 0.0268 & 7.439e-03 & 2.03 & 7.443e-03 & 2.03 & 1.275e-03 & 2.10\\
  & 6 & 4225 & 0.0136 & 1.897e-03 & 2.02 & 1.897e-03 & 2.02 & 3.188e-04 & 2.05\\
  & 7 & 16641 & 0.0069 & 4.672e-04 & 2.04 & 4.674e-04 & 2.04 & 7.951e-05 & 2.03\\
\midrule
2 & 4 & 289 & 0.0520 & 4.305e-02 & -- & 4.534e-02 & -- & 8.062e-03 & --\\
  & 5 & 1089 & 0.0268 & 7.890e-03 & 2.56 & 8.081e-03 & 2.60 & 1.127e-03 & 2.97\\
  & 6 & 4225 & 0.0136 & 1.113e-03 & 2.89 & 1.160e-03 & 2.86 & 1.436e-04 & 3.04\\
  & 7 & 16641 & 0.0069 & 1.333e-04 & 3.10 & 1.344e-04 & 3.14 & 1.850e-05 & 2.99\\
\midrule
3 & 4 & 289 & 0.0520 & 2.448e-02 & -- & 2.612e-02 & -- & 3.884e-03 & --\\
  & 5 & 1089 & 0.0268 & 5.446e-03 & 2.27 & 4.385e-03 & 2.69 & 3.209e-04 & 3.76\\
  & 6 & 4225 & 0.0136 & 4.829e-04 & 3.57 & 3.166e-04 & 3.88 & 2.386e-05 & 3.83\\
  & 7 & 16641 & 0.0069 & 3.260e-05 & 3.93 & 3.259e-05 & 3.32 & 1.535e-06 & 4.00\\
\midrule
4 & 4 & 289 & 0.0520 & 2.085e-02 & -- & 2.085e-02 & -- & 2.822e-03 & --\\
  & 5 & 1089 & 0.0268 & 1.711e-03 & 3.77 & 1.692e-03 & 3.79 & 1.298e-04 & 4.64\\
  & 6 & 4225 & 0.0136 & 6.545e-05 & 4.81 & 6.551e-05 & 4.80 & 4.977e-06 & 4.81\\
  & 7 & 16641 & 0.0069 & 1.799e-06 & 5.24 & 1.838e-06 & 5.21 & 1.603e-07 & 5.01\\
\bottomrule
\end{tabular}
\end{table}

\subsection{Approximation of piecewise smooth functions}

We consider the piecewise smooth datum
$\tilde f(x,y)=f(x,y)+\chi_{\Omega^+}(x,y)$, with $\chi_{\Omega^+}$ the
characteristic function of $\Omega^+$ and $f$ as in
\eqref{eq:franke} and $\Omega^+$ the disk of radius $0.25$ centered at
$(\tfrac12,\tfrac12)$, sampled at $N=65^2$ Halton nodes
($H=0.0246$).
Figure~\ref{fig:cenital} compares the original datum with the linear
and the nonlinear reconstructions for $d=1$; Figure~\ref{fig:errors}
shows the corresponding error maps in logarithmic scale, and
Figure~\ref{fig:section} a one-dimensional section together with the
empirical distribution of the smoothness indicators.

\begin{figure}[t]
\centering
\includegraphics[width=\linewidth]{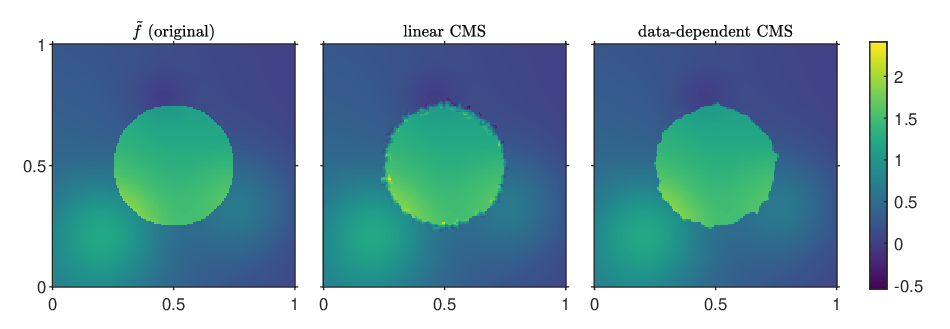}
\caption{Top views: original datum $\tilde f$ (left), linear CMS
(center), data-dependent CMS (right); circle geometry, Halton nodes,
$N=65^2$, $d=1$. }
\label{fig:cenital}
\end{figure}

\begin{figure}[t]
\centering
\includegraphics[width=\linewidth]{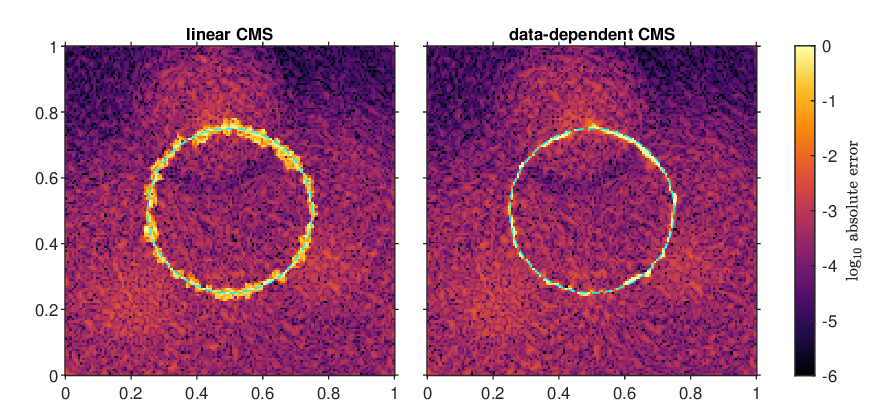}
\caption{$\log_{10}$ absolute error of the linear (left) and
data-dependent (right) CMS; the dashed curve is $\Gamma$. Same data as
Figure~\ref{fig:cenital}.}
\label{fig:errors}
\end{figure}

\begin{figure}[t]
\centering
\includegraphics[width=\linewidth]{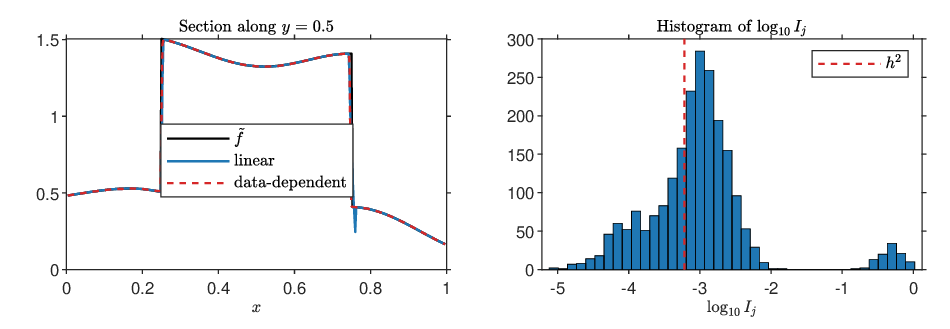}
\caption{Left: section along $y=0.5$. Right: histogram of
$\log_{10}\Ind{j}$; the dashed line marks $H^{2}$. 
Out of $2179$ stencils, the $65$ whose nodes straddle $\Gamma$
have indicators larger than $0.18$; together with $36$ one-sided
stencils near $\Gamma$, detected through their indicator sets as in
Lemma~\ref{lem:P1P2}(P2), they form the upper mode of the distribution,
$101$ stencils with $\Ind{j}>0.17$ (median $0.50$).
All the other indicators are at most $1.35\times10^{-2}$, so that no
indicator lies between $1.35\times10^{-2}$ and $0.17$; the two modes
of the distribution are separated by about three orders of magnitude.}
\label{fig:section}
\end{figure}

Table~\ref{tab:jump} quantifies the comparison for $d=1,\dots,4$. Three
effects are apparent. First, the error of the nonlinear operator at
distance $>H$ from $\Gamma$ decreases with the degree, within a factor
of four of the smooth case at distance $>H$ and at the level of the
smooth case at distance $>2H$, while the linear operator error remains between $0.25$ and $1.2$
there at the tested resolution. Second, the width of the contaminated band
(points with error larger than a quarter of the jump) for the linear
operator grows with the degree ($1.3H$, $1.8H$,
$2.0H$, $2.7H$), while it remains within $0.7$--$0.8H$ for the
nonlinear one. The band width is the largest distance from $\Gamma$ of an evaluation point
at which the error exceeds a quarter of the jump; it is measured on a
locally refined
set of evaluation points within $4H$ of $\Gamma$, with spacing $H/10$
in the normal direction and about $H/5$ or less along $\Gamma$. 
Third, the  deviations of the linear operator beyond the
range of $\tilde f$ on the evaluation grid remain substantial at all degrees, both above
the maximum (up to $+0.686$) and below the minimum (up to $0.954$),
 while for the nonlinear operator they are at most
$2.4\cdot10^{-2}$ above the maximum and $4.0\cdot10^{-3}$ below the
minimum.

To check that the width of the contaminated band is proportional
to $h$, we repeat the experiment on the four levels $l=4,\dots,7$ of
Table~\ref{tab:smooth}. Table~\ref{tab:bandref} reports the band widths,
measured on the refined set, in units of $H$: for the nonlinear operator
they remain between $0.6H$ and $0.8H$ at every level and degree, so that, over the tested levels,
the band shrinks proportionally to $H$; for the linear operator they
remain between $1.2H$ and $2.9H$. At the same four levels and for every
degree, the maximum error of the nonlinear operator at distance larger than
$2H$ from $\Gamma$ does not exceed its maximum error on the smooth datum
(Table~\ref{tab:smooth}), with the only exception of $d=3$, $l=5$, where it
is $1.2$ times as large; at distance larger than $H$ the ratio between the
two errors is at most $4$, except for $d=3$, $l=5$, where it is about $16$.

\begin{table}[t]
\centering\small
\caption{Jump on the circle, Halton nodes, $N=65^2$, $H=0.0246$:
maximum errors at distance $>H$ and $>2H$ from $\Gamma$, width of the
band where the error exceeds one quarter of the jump\ (measured on a
locally refined set near $\Gamma$),  signed
maximal deviation $\max\mathcal I-\max\tilde f$ (positive values are
overshoots; negative values indicate that the reconstruction stays
below the maximum of $\tilde f$), and undershoot
$(\min\tilde f-\min\mathcal I)_+$; maxima and minima are taken over the
evaluation grid.}
\label{tab:jump}
\begin{tabular}{clccccc}
\toprule
$d$ ($t$) & method & $\Einf$ ($>H$) & $\Einf$ ($>2H$) & band & deviation & undershoot\\
\midrule
1 (4)  & linear & 3.344e-01 & 6.255e-03 & $1.3H$ & $+5.79$e-01 & 2.55e-01\\
       & WENO   & 7.280e-03 & 6.174e-03 & $0.7H$ & $-8.82$e-04 & 4.00e-03\\
\midrule
2 (6)  & linear & 2.545e-01 & 3.172e-02 & $1.8H$ & $+2.71$e-01 & 1.84e-01\\
       & WENO   & 1.665e-03 & 8.619e-04 & $0.8H$ & $+3.95$e-04 & 0\\
\midrule
3 (8)  & linear & 1.150e+00 & 2.334e-01 & $2.0H$ & $+3.99$e-01 & 9.19e-01\\
       & WENO   & 5.096e-04 & 2.352e-04 & $0.8H$ & $+2.38$e-02 & 0\\
\midrule
4 (12) & linear & 7.318e-01 & 5.577e-01 & $2.7H$ & $+6.86$e-01 & 9.54e-01\\
       & WENO   & 3.656e-04 & 9.274e-05 & $0.8H$ & $-6.51$e-03 & 0\\
\bottomrule
\end{tabular}
\end{table}

\begin{table}[t]
\centering\small
\caption{Jump on the circle, Halton nodes, $N=(2^l+1)^2$: width
of the band where the error exceeds one quarter of the jump, in units
of $H$, for the linear and the data-dependent (WENO) CMS (measured as in
Table~\ref{tab:jump}).}
\label{tab:bandref}
\begin{tabular}{cl cccc}
\toprule
$d$ ($t$) & method & $l=4$ & $l=5$ & $l=6$ & $l=7$\\
 & & $H=0.0879$ & $H=0.0421$ & $H=0.0246$ & $H=0.0129$\\
\midrule
1 (4)  & linear & 1.2 & 1.6 & 1.3 & 1.3\\
       & WENO   & 0.7 & 0.7 & 0.7 & 0.7\\
\midrule
2 (6)  & linear & 1.6 & 2.1 & 1.8 & 1.8\\
       & WENO   & 0.6 & 0.7 & 0.8 & 0.6\\
\midrule
3 (8)  & linear & 1.5 & 2.9 & 2.0 & 2.3\\
       & WENO   & 0.7 & 0.8 & 0.8 & 0.8\\
\midrule
4 (12) & linear & 2.5 & 2.4 & 2.7 & 2.1\\
       & WENO   & 0.6 & 0.8 & 0.8 & 0.7\\
\bottomrule
\end{tabular}
\end{table}

To assess the dependence on the geometry of $\Gamma$, we repeat
the experiment with $\Omega^+=\{x+y\le1\}$, whose boundary $\Gamma$ is a
line crossing $\partial\Omega$, and with $\Omega^+=[\tfrac12,1]^2$, a square
with a corner at $(\tfrac12,\tfrac12)$ whose sides meet $\partial\Omega$,
using the same nodes and parameters; the results are reported in
Table~\ref{tab:geom}.\ The corner of the square lies outside the
$C^1$ setting of Section~\ref{sec:properties}, so that this geometry
serves as a stress test. At distance greater than $2H$ from $\Gamma$ the
nonlinear operator recovers, for both geometries and every degree, the
accuracy of the smooth case: the errors do not exceed the maximum
errors of Table~\ref{tab:smooth} at $l=6$. The linear operator, in contrast, has
errors comparable to the jump at distance greater than $H$ in all cases
and, for the line with $d\ge2$, also beyond $2H$. 
Between $H$ and $2H$ the nonlinear operator is less robust than
for the circle. Errors larger than a quarter of the jump occur for the
line with $d=2$ and for the square with $d=3$ and $d=4$, at no more than
two evaluation points in each case, all at distance between $H$ and
$1.5H$ from $\Gamma$; for the line with $d=3$ and $d=4$ the error at
distance $>H$ remains below a quarter of the jump but well above the
smooth level, while for the square with $d=2$ there is no
deterioration. The inspection of the
active stencils at these points shows that they are precisely
configurations excluded by Assumption~\ref{ass:separation}. Near the
points where the line meets $\partial\Omega$, only stencils straddling
$\Gamma$ are active, so that (A2) fails. Near the corner of the square
and near the points where its sides meet $\partial\Omega$, a one-sided
stencil of the opposite branch is active and, its indicator being
small, receives a large weight, so that (A1) fails, as anticipated in
the comment following Assumption~\ref{ass:separation}. In all these
cases the straddling stencils are detected by the indicator
($\Ind{j}\ge2.6\cdot10^{-2}$): the failures are due to the covering and
not to the detection mechanism, in accordance with
Remark~\ref{rem:scope}.

Finally, we compare the Halton set with a regular grid of the
same cardinality. Distances from $\Gamma$ and band widths have so far
been expressed in units of $H$. For a fixed number of nodes, however,
$H$ depends on the uniformity of the node distribution, whereas the
size of the stencils depends on the node density: for $N=65^2$, $H$ is
about $1.6$ times the mean node spacing on the Halton set and about
$0.87$ times the spacing on the regular grid. Thresholds in units of
$H$ are therefore not comparable across node distributions, and we
measure the errors of the nonlinear operator at the same physical
distances $\delta_H$, $2\delta_H$ and $3\delta_H$ from $\Gamma$, with
$\delta_H=0.0246$ the value of $H$ for the Halton set, for the three
geometries of this section (Table~\ref{tab:physical}). For $d\le3$ the
two node distributions behave comparably: at distance larger than
$2\delta_H$ the regular grid yields, for every geometry, errors
comparable to those of the smooth case (Table~\ref{tab:smoothgrid},
$l=6$). For $d=4$ the contaminated band on the grid is wider: errors
of order $10^{-2}$ persist up to $2\delta_H$ for all three geometries,
and the smooth level is reached only at distance $3\delta_H$; closer to
$\Gamma$, for the circle, isolated values exceed the jump, with errors
up to about $4$ at distance between the grid value of $H$ and
$\delta_H$. At distance larger than $\delta_H$ the grid is more accurate
than the Halton set for $d=1$ and in some cases with $d\ge2$ (the line
with $d=2$ and $d=3$, the square with $d=3$ and $d=4$), and less accurate
in the others (the circle with $d=2,3,4$, the line with $d=4$, the
square with $d=2$). Two features specific to regular grids should be
noted. First, the discrete Leja selection meets exact ties, so that the
stencils, and hence the errors closest to $\Gamma$, depend on the
tie-breaking rule of the LU factorization. Second, for the line and the
square some nodes lie exactly on $\Gamma$. Shifting $\Gamma$ by $0.37$
grid spacings, so that no node lies on it, reduces the error at
distance larger than $\delta_H$ for the square with $d=3$ from $0.17$ to
$5.8\cdot10^{-4}$ and for the square with $d=2$ from $3.0\cdot10^{-3}$
to $9.9\cdot10^{-4}$, while it leaves the other cases essentially
unchanged. Overall, on regular grids the method performs as on
scattered nodes for $d\le3$, while for $d=4$ the contaminated band is
wider.

\begin{table}[t]
\centering\small
\caption{Jump along a line ($\Omega^+=\{x+y\le1\}$) and along the
sides of a square ($\Omega^+=[\frac12,1]^2$), Halton nodes, $N=65^2$,
$H=0.0246$: maximum errors at distance $>H$ and $>2H$ from $\Gamma$ and
width of the band where the error exceeds one quarter of the jump
(measured as in Table~\ref{tab:jump}; for the linear operator on the square
with $d=2$, the $100^2$ evaluation grid contains a point above the threshold
at distance $1.4H$ from $\Gamma$, not captured by the refined set).}
\label{tab:geom}
\begin{tabular}{cl ccc ccc}
\toprule
 & & \multicolumn{3}{c}{line} & \multicolumn{3}{c}{square}\\
\cmidrule(lr){3-5}\cmidrule(lr){6-8}
$d$ ($t$) & method & $\Einf$ ($>H$) & $\Einf$ ($>2H$) & band & $\Einf$ ($>H$) & $\Einf$ ($>2H$) & band\\
\midrule
1 (4)  & linear & 4.345e-01 & 6.255e-03 & $1.5H$ & 4.493e-01 & 6.255e-03 & $1.4H$\\
       & WENO   & 6.174e-03 & 6.174e-03 & $0.6H$ & 6.174e-03 & 6.174e-03 & $0.5H$\\
\midrule
2 (6)  & linear & 1.026e+00 & 1.026e+00 & $2.0H$ & 2.854e-01 & 9.474e-04 & $1.0H$\\
       & WENO   & 4.000e-01 & 8.636e-04 & $1.4H$ & 9.541e-04 & 8.619e-04 & $0.8H$\\
\midrule
3 (8)  & linear & 1.437e+00 & 1.437e+00 & $2.1H$ & 7.482e-01 & 7.912e-02 & $1.7H$\\
       & WENO   & 5.555e-02 & 2.352e-04 & $0.8H$ & 1.000e+00 & 2.352e-04 & $1.1H$\\
\midrule
4 (12) & linear & 2.183e+00 & 4.095e-01 & $2.2H$ & 6.035e-01 & 2.345e-01 & $2.5H$\\
       & WENO   & 1.380e-01 & 9.274e-05 & $1.0H$ & 3.528e-01 & 9.274e-05 & $1.0H$\\
\bottomrule
\end{tabular}
\end{table}

\begin{table}[t]
\centering\small
\caption{Halton nodes versus regular grid, $N=65^2$: maximum error
of the data-dependent CMS at distance larger than $\delta_H$, $2\delta_H$
and $3\delta_H$ from $\Gamma$, with $\delta_H=0.0246$ for both node
distributions.}
\label{tab:physical}
\setlength{\tabcolsep}{4pt}
\begin{tabular}{cc ccc ccc}
\toprule
 & & \multicolumn{3}{c}{Halton} & \multicolumn{3}{c}{regular grid}\\
\cmidrule(lr){3-5}\cmidrule(lr){6-8}
$\Gamma$ & $d$ & $>\delta_H$ & $>2\delta_H$ & $>3\delta_H$ & $>\delta_H$ & $>2\delta_H$ & $>3\delta_H$\\
\midrule
circle & 1 & 7.3e-03 & 6.2e-03 & 6.2e-03 & 1.9e-03 & 1.8e-03 & 1.8e-03\\
 & 2 & 1.7e-03 & 8.6e-04 & 8.3e-04 & 1.9e-03 & 1.2e-03 & 1.2e-03\\
 & 3 & 5.1e-04 & 2.4e-04 & 2.4e-04 & 8.2e-04 & 2.9e-04 & 2.9e-04\\
 & 4 & 3.7e-04 & 9.3e-05 & 9.3e-05 & 9.0e-02 & 4.3e-02 & 5.2e-05\\
\midrule
line & 1 & 6.2e-03 & 6.2e-03 & 6.2e-03 & 1.9e-03 & 1.9e-03 & 1.9e-03\\
 & 2 & 4.0e-01 & 8.6e-04 & 8.6e-04 & 1.1e-01 & 1.2e-03 & 1.2e-03\\
 & 3 & 5.6e-02 & 2.4e-04 & 2.4e-04 & 3.9e-02 & 3.2e-04 & 3.2e-04\\
 & 4 & 1.4e-01 & 9.3e-05 & 9.3e-05 & 2.6e-01 & 2.2e-02 & 6.6e-05\\
\midrule
square & 1 & 6.2e-03 & 6.2e-03 & 6.2e-03 & 1.9e-03 & 1.9e-03 & 1.8e-03\\
 & 2 & 9.5e-04 & 8.6e-04 & 8.6e-04 & 3.0e-03 & 9.9e-04 & 9.9e-04\\
 & 3 & 1.0e+00 & 2.4e-04 & 2.4e-04 & 1.7e-01 & 3.1e-04 & 3.1e-04\\
 & 4 & 3.5e-01 & 9.3e-05 & 9.3e-05 & 1.2e-01 & 4.0e-02 & 5.2e-05\\
\bottomrule
\end{tabular}
\end{table}

\subsection{Sensitivity to the exponent $t$}

Table~\ref{tab:tsens} reports the maximum error at distance $>H$ from
$\Gamma$ as a function of $t$, for $d=2$ and $d=4$. For $d=2$ the
choice $t=6$ is the best among those tested, with a mild degradation
for larger $t$. For $d=4$ the error decreases monotonically up to
$t=12$\ and then remains at the same level
for $t=14$ and $t=16$, in qualitative agreement with the discussion
of the exponent in Section~\ref{subsec:jump}.

\begin{table}[t]
\centering\small
\caption{Sensitivity to $t$: maximum error at distance $>H$ from
$\Gamma$ (circle geometry, Halton nodes, $N=65^2$).}
\label{tab:tsens}
\setlength{\tabcolsep}{4pt}
\begin{tabular}{ccccccccc}
\toprule
 & $t=2$ & $t=4$ & $t=6$ & $t=8$ & $t=10$ & $t=12$ & $t=14$ & $t=16$\\
\midrule
$d=2$ & 5.828e-02 & 2.864e-02 & 1.665e-03 & 2.194e-03 & 2.778e-03 & 3.014e-03 & -- & --\\
$d=4$ & 5.140e-01 & 1.834e-01 & 7.677e-02 & 5.325e-02 & 7.434e-03 & 3.656e-04 & 3.669e-04 & 3.674e-04\\
\bottomrule
\end{tabular}
\end{table}

\subsection{Reproducibility}
The experiments were performed with a
MATLAB code building on the multinode Shepard implementation
\cite{DellAccioSoftware}; the results were independently verified with a separate Python
implementation. The codes, together with the scripts generating every
table and figure of this section, are available from the authors.

\section{Conclusions}\label{sec:conclusions}
We have introduced a data-dependent compact multinode Shepard method
which combines the approximation order of the multinode construction
with the essentially non-oscillatory behaviour of WENO-type weights.
The smoothness indicator is provided by the residuals of the local
interpolants on nodes outside the stencils, without additional function
evaluations, and the nonlinear renormalization preserves the partition
of unity, the interpolation property, the guaranteed regularity class
and the convex-combination structure of the linear operator. The
operator attains the order $h^{d+1}$ in smooth regions for any choice
of the parameters and, under an explicit indicator-separation
hypothesis, the same order near a jump discontinuity, with a
transparent analysis of the constants: the prefactor of the damped
contributions grows exponentially with the stencil cardinality, which
may contribute to the finite-resolution need for larger exponents $t$
at higher degrees. We have also delimited the scope of the detection mechanism by
an explicit configuration in which a stencil straddling the
discontinuity escapes the basic indicator. The numerical experiments confirm the theory
for local degrees up to four, on scattered and on regular grid nodes. For
the circular discontinuity on Halton nodes with $N=65^2$ the nonlinear
operator improves on the linear one, at distance larger than $H$ from
$\Gamma$, by a factor between $45$ and $2.3\cdot10^{3}$
(Table~\ref{tab:jump}), with the largest gains at the higher degrees. For
the line and square geometries the corresponding gain is irregular in the
degree, ranging from a factor of about $300$ (square, $d=2$) to none
(square, $d=3$, where the nonlinear operator is the less accurate of the
two), and on regular grids the contaminated band is wider for $d=4$
(Section~\ref{sec:numerics}); for all the geometries and degrees tested on
Halton nodes, with the exponents $t$ of Section~\ref{subsec:jump}, the nonlinear
operator recovers the accuracy of the smooth case at distance larger than
$2H$.

The few points at distance larger than $H$ from $\Gamma$ where the
nonlinear operator loses accuracy, on the line and square geometries, are
all points where either no one-sided stencil is active or a one-sided
stencil of the opposite branch is active (Section~\ref{sec:numerics}):
the loss is due to the covering, not to the indicator. This, together
with the exponential dependence of the constants on the stencil
cardinality, points to coverings adapted to the discontinuity, in which
straddling stencils are never active and conditions (A1) and (A3) of
Assumption~\ref{ass:separation} hold by construction. This is the subject
of a forthcoming paper by the authors.

\begin{acknowledgements}
The first three authors carried out this research as part of
RITA (Research ITalian network on Approximation) and the
UMI group ``Teoria dell'Approssimazione e Applicazioni''.
They have been supported by the GNCS-INdAM 2026 project
``Metodi polinomiali e kernel per l'approssimazione da dati
discreti e integrali con software OS''.
They are also members of the INdAM-GNCS Research Group.

The fourth and fifth authors have been supported through GVA project CIAICO/2024/089, the fifth author has also been funded by grant  PID2023-146836NB-I00 funded by MCIN/AEI/10.13039/501100011033 and by ERDF/EU.
\end{acknowledgements}

\bibliographystyle{elsarticle-num}
\bibliography{references.bib}

\end{document}